\documentclass[11pt]{amsart}
\usepackage[colorlinks = true,
            linkcolor = red,
            urlcolor  = magenta,
            citecolor = black,
            anchorcolor = blue]{hyperref}
\usepackage{color}
\usepackage[shortalphabetic,initials,nobysame]{amsrefs}
\usepackage{upgreek}
\usepackage{tikz}
\usepackage[applemac]{inputenc} 
\usetikzlibrary{arrows}
\usepackage{xcolor}
\usepackage[all]{xy}
\usepackage{graphicx}
\usepackage{wrapfig}
\usepackage{yfonts}
\usepackage{graphicx}
\usepackage{amssymb}
\usepackage{epstopdf}
\usepackage{mathrsfs}
\usepackage[mathcal]{eucal}
\usepackage{bm}

\renewcommand{\eprint}[1]{\href{https://arxiv.org/abs/#1}{#1}}

\DeclareMathOperator{\GL}{\mathrm{GL}}
\DeclareMathOperator{\PGL}{\mathrm{PGL}}
\DeclareMathOperator{\SL}{\mathrm{SL}}
\DeclareMathOperator{\Span}{span}
\DeclareMathOperator{\diag}{diag}
\DeclareMathOperator{\Hom}{Hom}
\DeclareMathOperator{\Res}{Res}

\newcommand{\tz}{\tilde{z}}

\newcommand{\hcL}{\hat{\cL}}
\newcommand{\om}{\omega}
\newcommand{\hf}{\hat{f}}
\newcommand{\bfa}{\bm{a}}

\newcommand{\fg}{\frak{g}}
\newcommand{\blam}{\bm{\lambda}}
\newcommand{\lam}{\lambda}
\newcommand{\mfq}{\mathfrak{q}}
\newcommand{\sD}{\mathscr{D}}

\let\oldmarginpar\marginpar
\renewcommand\marginpar[1]{\-\oldmarginpar[\raggedleft\footnotesize #1]
{\raggedright\footnotesize #1}}

\BibSpec{article}{%
    +{}  {\PrintAuthors}                {author}
    +{,} { \textit}                     {title}
    +{.} { }                            {part}
    +{:} { \textit}                     {subtitle}
    +{,} { \PrintContributions}         {contribution}
    +{.} { \PrintPartials}              {partial}
    +{,} { }                            {journal}
    +{}  { \textbf}                     {volume}
    +{ } { \PrintDatePV}            {date}
    +{,} { \issuetext}                  {number}
    +{,} { \eprintpages}                {pages}
    +{,} { }                            {status}
    +{,} { \DOI}                   {doi}
    +{,} { \eprint}        {eprint}
    +{}  { \parenthesize}               {language}
    +{}  { \PrintTranslation}           {translation}
    +{;} { \PrintReprint}               {reprint}
    +{.} { }                            {note}
    +{.} {}                             {transition}
    +{}  {\SentenceSpace \PrintReviews} {review}
}

\newtheorem{theorem}{Theorem}[section]
\newtheorem{lemma}[theorem]{Lemma}

\newtheorem{Thm}{Theorem}[section]
\newtheorem{Lem}[Thm]{Lemma}

\theoremstyle{definition}
\newtheorem{Def}[Thm]{Definition}
\theoremstyle{remark}
\newtheorem{Rem}[Thm]{Remark}

\newtheoremstyle{named}{}{}{\itshape}{}{\bfseries}{.}{.5em}{#1 #3}
\theoremstyle{named}

\def\Q{\mathbb{Q}}

\def\C{\mathbb{C}}
\def\Z{\mathbb{Z}}

\def\P{\mathbb{P}}

\def\sl{\mathfrak{sl}}
\def\gl{\mathfrak{gl}}

\def\cD{\mathcal{D}}
\def\cE{\mathcal{E}}
\def\cF{\mathcal{F}}
\def\cG{\mathcal{G}}
\def\cH{\mathcal{H}}

\def\cL{\mathcal{L}}

\def\cO{\mathcal{O}}

\def\cV{\mathcal{V}}
\def\cW{\mathcal{W}}

\def\cZ{\mathcal{Z}}

\def\a{\alpha}

\def\ze{\zeta}

\def\s{\sigma}

\def\bo{\textbf{o}}

\def\=>{\Longrightarrow}

\def\to{\longrightarrow}

\def\o+{\oplus}
\def\bo+{\bigoplus}

\def\<{\langle}
\def\>{\rangle}
\def\({\left(}
\def\){\right)}

\def\^{\wedge}
\def\+{\dagger}

\def\half{\frac{1}{2}}

\def\dd[#1,#2]{\frac{d#1}{d#2}}
\def\del[#1,#2]{\frac{\partial #1}{\partial #2}}
\def\over[#1]{\overline{#1}}
\def\vec[#1]{\overrightarrow{#1}}

\makeatletter
\def\mr@ignsp#1 {\ifx\:#1\@empty\else #1\expandafter\mr@ignsp\fi}%
\newcommand{\multiref}[1]{\begingroup
\xdef\mr@no@sparg{\expandafter\mr@ignsp#1 \: }%
\def\mr@comma{}%
\@for\mr@refs:=\mr@no@sparg\do{\mr@comma\def\mr@comma{,}\ref{\mr@refs}}%
\endgroup}
\makeatother

\newcommand{\hypref}[2]{\ifx\href\asklfhas #2\else\href{#1}{#2}\fi}
\newcommand{\Secref}[1]{Section~\multiref{#1}}

\newcommand{\figref}[1]{Fig.~\multiref{#1}}

\tikzset{->-/.style={decoration={
  markings,
  mark=at position .5 with {\arrow{latex}}},postaction={decorate}}}
\tikzset{
    >=latex
    }

\begin{document}
\title[$(\SL(N),q)$-opers, $q$-Langlands correspondence, quantum/classical duality]{$(\SL(N),q)$-opers, the $q$-Langlands correspondence, and quantum/classical duality}

\author{Peter Koroteev}
\address[Peter Koroteev]{\newline
Department of Mathematics,\newline
University of California at Davis,\newline
Mathematical Sciences Building,\newline
One Shields Ave,\newline
University of California,\newline
Davis, CA 95616\newline
         Email: pkoroteev@math.ucdavis.edu
          \newline \href{https://www.math.ucdavis.edu/research/profiles/?fac_id=pkoroteev}{URL}}

\author{Daniel S. Sage}
\address[Daniel S. Sage]{\newline
          Department of Mathematics, \newline
          Louisiana State University, \newline
          Baton Rouge, LA, USA; \newline
          Email: sage@math.lsu.edu
          \newline \href{http://math.lsu.edu/\textasciitilde sage}{URL}}

\author{Anton M. Zeitlin}
\address[Anton M. Zeitlin]{\newline
          Department of Mathematics, \newline
          Louisiana State University, \newline
          Baton Rouge, LA, USA; \newline
          IPME RAS, St. Petersburg
          \newline
          Email: zeitlin@lsu.edu
          \newline \href{http://math.lsu.edu/\textasciitilde zeitlin}{URL}}

\date{\today}

\numberwithin{equation}{section}

\maketitle

\begin{abstract}
  A special case of the geometric Langlands correspondence is given by
  the relationship between solutions of the Bethe ansatz equations for
  the Gaudin model and opers--connections on the projective line with
  extra structure.  In this paper, we describe a deformation of this
  correspondence for $\SL(N)$.  We introduce a difference equation
  version of opers called $q$-opers and prove a $q$-Langlands
  correspondence between nondegenerate solutions of the Bethe ansatz
  equations for the XXZ model and nondegenerate twisted $q$-opers with
  regular singularities on the projective line.  We show that the
  quantum/classical duality between the XXZ spin chain and the
  trigonometric Ruijsenaars-Schneider model may be viewed as a special
  case of the $q$-Langlands correspondence.  We also describe an
  application of $q$-opers to the equivariant quantum $K$-theory of
 the cotangent bundles to partial flag varieties.

\end{abstract}

\setcounter{tocdepth}{1}
\tableofcontents
\section{Introduction}\label{sec:intro}

\subsection{Opers and the Gaudin model}

One formulation of the geometric Langlands correspondence is the
existence of an isomorphism between spaces of conformal blocks for the
classical $W$-algebra associated to a simple complex Lie algebra $\fg$
and the dual affine Kac-Moody algebra ${}^L\hat{\fg}$ at the critical
level.  Since both these algebras admit deformations, it is natural to
conjecture the existence of deformed versions of the Langlands
correspondence, and indeed, this has been the subject of considerable
recent interest~\cite{Aganagic:2017smx,Gaiotto:2018,PestunStringTalk}.
In this paper, we describe a $q$-Langlands correspondence which is a
deformation of an important example of geometric Langlands, the
classical correspondence between the spectra of the Gaudin model and
opers on the projective line with regular singularities and trivial
monodromy.

Let $G$ be a simple complex algebraic group of adjoint type, and let
${}^L\fg$ be the Lie algebra of the Langlands dual group ${}^L G$.  Fix a
collection of distinct points $z_1,\dots,z_n$ in $\C$.  The Gaudin
Hamiltonians are certain mutually commuting elements of the algebra
$U({}^L \fg)^{\otimes n}$.  They are contained in a commutative
subalgebra $\cZ_{(z_i)}({}^L \fg)$ called the Gaudin algebra.  The
simultaneous eigenvalues of the actions of the Gaudin Hamiltonians on
$N$-fold tensor products of ${}^L\fg$-modules is given by the (maximal)
spectrum of this algebra, namely, the set of algebra homomorphisms
$\cZ_{(z_i)}({}^L \fg)\to\C$.\footnote{We remark that the Gaudin
  algebra is a subalgebra of the center of the universal enveloping
  algebra at the critical level, and this center was characterized in
  \cite{Feigin:1994in}. Ding and Etingof have found generators of the center of
  the quantum affine algebra at the critical level, thereby proving the $q$-analogue of this
  statement~\cite{Ding_1994}.}

Feigin, Frenkel, and Reshetikhin found a geometric interpretation of
this spectrum in terms of flat $G$-bundles on $\P^1$ with extra
structure~\cite{Feigin:1994in,Frenkel:ab,Frenkel:aa}. Let $B$ be a
Borel subgroup of $G$.  A $G$-oper on a smooth curve $X$ is a triple
$(\cF,\nabla,\cF_B)$, where $(\cF,\nabla)$ is a flat $G$-bundle on $X$
and $\cF_B$ is a reduction of $\cF$ satisfying a certain
transversality condition with respect to $\nabla$.  As an example, for
$\PGL(2)$-opers, this condition is that $\cF_B$ is nowhere preserved by
$\nabla$.  The space of $G$-opers can be realized more concretely as a
certain space of differential operators.  For example, a $\PGL(2)$-oper
can be identified with projective connections: second-order operators
$\partial_z^2-f(z)$ mapping sections of $K^{-1/2}$ to sections of
$K^{3/2}$, where $K$ is the canonical bundle.  It turns out that the
spectrum of $\cZ_{(z_i)}({}^L\fg)$ may be identified with the set of
$G$-opers on $\P^1$ with regular singularities at $z_1,\dots,z_n$ and
$\infty$, which therefore has the structure of an algebraic variety.

We now consider the action of the Gaudin algebra on the tensor product
of irreducible finite-dimensional modules
$V_{\blam}=V_{\lam_1}\otimes\dots \otimes V_{\lam_n}$, where $\blam$ is an
$n$-tuple of dominant integral weights.  The Bethe ansatz is a method
of constructing such simultaneous eigenvectors.  One starts with the
unique (up to scalar) vector $|0\rangle\in V_{\blam}$ of highest
weight $\sum\lam_i$; it is a simultaneous eigenvector.  Given a set of
distinct complex numbers $w_1,\dots,w_m$ labeled by simple roots
$\alpha_{k_j}$, one applies a certain order $m$ lowering operator with
poles at the $w_j$'s to $|0\rangle$.  If this vector is nonzero and
$\sum \lam_i -\sum\a_{k_j}$ is dominant, it is an eigenvector of the
  Gaudin Hamiltonians if and only if certain equations called the
  Bethe ansatz equations are satisfied (see \eqref{Betheeqsln}).
  Frenkel has shown that the corresponding point in the spectrum of
  the Gaudin algebra is a $G$-oper with regular singularities at the
  $z_i$'s and $\infty$ and and with trivial
  monodromy~\cite{Frenkel:aa}.

  In fact, it is possible to give a geometric description of all
  solutions of the Bethe equations (i.e., without assuming $\sum
  \lam_i -\sum\a_{k_j}$ is dominant) in terms of an enhanced version
  of opers.  A \emph{Miura $G$-oper} on $\P^1$ is a $G$-oper together
  with an additional reduction $\cF'_B$ which is preserved by
  $\nabla$.  The set of Miura opers with the same underlying oper is
  parametrized by the flag manifold $G/B$.  Frenkel has shown that
  there is a one-to-one correspondence between the set of solutions to
  the ${}^L \fg$ Bethe ansatz and ``nondegenerate'' Miura $G$-opers
  with regular singularities and trivial monodromy~\cite{Frenkel:ab}.
  To see how this works, let $H\subset B$ be a maximal torus.  The
  initial data of the Bethe ansatz gives rise to the explicit flat
  $H$-bundle (a \emph{Cartan connection})
\begin{equation*} \partial_z+\sum_{i=1}^n \frac{\lam_i}{z-z_i}-\sum_{j=1}^m\frac{\a_{k_j}}{z-w_j}.
\end{equation*}
There is a map from Cartan connections to Miura opers given by the
Miura transformation; this is just a generalization of the standard
Miura transformation in the theory of KdV integrable models.  It turns
out that the Bethe equations are precisely the conditions necessary
for the corresponding Miura oper to be regular at the $w_j$'s.

In the global geometric Langlands correspondence for $\P^1$, the
objects on the Galois side are flat $G$-bundles (with singularities)
on $\P^1$ while on the automorphic side, one considers $D$-modules on
enhanced versions of the moduli space of ${}^L G$-bundles over $\P^1$.
The correspondence between opers and spectra of the Gaudin model
provides an example of geometric Langlands.  Indeed, the eigenvector
equations for the Gaudin Hamiltonians for fixed eigenvalues determines
a $D$-module on the moduli space of ${}^L G$-bundles with parabolic
structures at $z_1,\dots,z_n$ and $\infty$ while the oper gives the
flat $G$-bundle.

\subsection{ $q$-opers and the $q$-Langlands correspondence}

Recall that the geometric Langlands correspondence may be viewed as an
identification of conformal blocks for the classical $\cW$-algebra
associated to $\fg$ and conformal blocks for the affine Kac-Moody
algebra ${}^L\hat{\fg}$ at the critical level.  Both these algebras
admit deformations.  For example, one may pass from ${}^L \hat{\fg}$
to the associated quantum affine algebra while at the same time moving
away from the critical level.  This led Aganagic, Frenkel, and
Okounkov to formulate a two-parameter deformation of geometric
Langlands called the quantum $q$-Langlands
correspondence~\cite{Aganagic:2017smx}.  This is an identification of
certain conformal blocks of a quantum affine algebra with those of a
deformed $\cW$-algebra, working over the infinite cylinder.  They
prove this correspondence in the simply-laced case; their proof is
based on a study of the equivariant $K$-theory of Nakajima quiver
varieties whose quiver is the Dynkin diagram of $\fg$.

In this paper, we take another more geometric approach, involving
\emph{$q$-connections}, a difference equation version of flat
$G$-bundles.  Our goal is to establish a $q$-Langlands correspondence
between \emph{$q$-opers} with regular singularities and the spectra of
the XXZ spin chain model. Here, we only consider this correspondence
in type $A$, where we can describe $q$-opers explicitly as a vector
bundle endowed with a difference operator, together with a complete
flag of subbundles that is well-behaved with respect to the
operator.\footnote{The authors and E. Frenkel have subsequently solved
  this problem in the general case using different
  techniques~\cite{Frenkel:2020a}.}

Fix a nonzero complex number $q$ which is not a root of unity.  We are
interested in (multiplicative) difference equations of the form
$s(qz)=A(z)s(z)$; here $A(z)$ is an $N\times N$ invertible matrix
whose entries are rational functions.  To express this more
geometrically, we start with a trivializable rank $n$ vector bundle
$E$ on $\P^1$, and let $E^q$ denote the pullback of $E$ via the map
$z\mapsto qz$. A $(\GL(N),q)$-connection on $\P^1$ is an invertible
operator $A$ taking sections of $E$ to sections of $E^q$.  If the
matrices $A(z)$ have determinant one in some trivialization, $(E,A)$
is called an $(\SL(N),q)$-connection.  Just as in the classical
setting, an $(\SL(N),q)$-oper is a triple $(E,A,E_B)$, where $E_B$ is a
reduction to a Borel subgroup satisfying a certain transversality
condition with respect to $A$.  We also define a Miura $q$-oper to be
a $q$-oper with an additional reduction $E'_B$ preserved by $A$.  We
remark that these definitions make sense when $\P^1$ is replaced by
the formal punctured disk.  In this setting, a concept equivalent to
$(\GL(N),q)$-connections was introduced by Baranovsky and
Ginzburg~\cite{Baranovsky:1996} while the notion of a formal $q$-oper
is inherent in the work of Frenkel, Reshetikhin, Semenov-Tian-Shansky,
and Sevostyanov on Drinfeld-Sokolov reduction for difference
operators~\cite{Frenkel_1998,1998CMaPh.192..631S}.

We now explain how $q$-opers can be viewed as the Galois side of a
$q$-Langlands correspondence.  The XXZ spin chain model is an
integrable model whose dynamical symmetry algebra is the quantum
affine algebra $U_q(\hat{\fg})$~\cite{Reshetikhin:1987}.  Under an
appropriate limiting process, it degenerates to the Gaudin model.  The
model depends on certain twist parameters which can be described by a
diagonal matrix $Z$.  We will always assume that $Z$ has distinct
eigenvalues.  Eigenvectors of the Hamiltonians in the XXZ model can
again be found using the Bethe ansatz, and the spectra can be
expressed in terms of Bethe equations (see \eqref{sl2qbethe},
\eqref{eq:BetheExplicit} below).

It turns out that these equations also arise from appropriate
$q$-opers.  We consider $q$-opers with regular singularities on
$\P^1\setminus\{0,\infty\}$.  We further assume that the $q$-oper is
\emph{$Z$-twisted}, where $Z$ is the diagonal matrix appearing in the
Bethe equations; this simply means that the underlying $q$-connection
is $q$-gauge equivalent to the $q$-connection with matrix $Z$.  (This
may be viewed as the quantum analogue of the opers with a double pole
singularity at $\infty$ considered by Feigin, Frenkel, Rybnikov, and
Toledano-Laredo in their work on an inhomogeneous version of the
Gaudin model~\cite{FFTL:2010,Rybnikov:2010}.)  Given a $Z$-twisted
$q$-oper with regular singularities, we examine a certain associated
Miura $q$-oper.  The assumption that this Miura $q$-oper is
``nondegenerate'' imposes certain conditions on the zeros of
\emph{quantum Wronskians} arising from the $q$-oper, and these
conditions lead to the XXZ Bethe equations.  Thus, in type $A$, we
obtain the desired $q$-Langlands correspondence.  It should be noted
that in contrast to the results of \cite{Aganagic:2017smx}, our results
do not depend on geometric data related to the quantum $K$-theory of
Nakajima quiver varieties.  In particular, there are no restrictions
on the dominant weights that can appear in our correspondence.

Our approach has some similarities with the earlier work of Mukhin and
Varchenko on discrete opers and the spectra of the XXX
model~\cite{Mukhin_2005}.  Here, they considered additive difference
equations, i.e., equations of the form $f(z+h)=A(z)f(z)$ where $A$ is
a $G$-valued function and $h\in\C^*$ is a fixed parameter.  They
defined a discrete oper to be the linear difference operator
$f(z)\mapsto f(z+h)-A(z)f(z)$ if $A(z)$ had a suitable form.  They
also introduced a notion of discrete Miura oper and showed that they
correspond to solutions of the XXX Bethe ansatz equations.  Unlike our
$q$-opers, these discrete opers do not seem to be related to the
difference equation version of Drinfeld-Sokolov reduction considered
in~\cite{Frenkel_1998}.

Since the XXZ model may be viewed as a deformation of the Gaudin
model, one would expect that we should recover the Gaudin Bethe
equations under an appropriate limit.  In fact, by taking this limit
in two steps, one can say more.  First, a suitable limit takes one to
a twisted version of the XXX spin chain, giving rise to a
correspondence between the solutions of the Bethe equations for this
model and a twisted analogue of the discrete opers of
\cite{Mukhin_2005}. A further limit brings one back to the
inhomogeneous Gaudin model and opers with irregular singularity
considered in ~\cite{FFTL:2010,Rybnikov:2010}.

\subsection{QQ-systems, XXZ models, and Baxter operators}

As we have seen, Bethe ansatz equations arise from $q$-opers by
considering quantum Wronksians and their properties.   More
precisely, the algebraic structure of the set of 
quantum Wronksians governs both twisted $Z$-opers and solutions of
the XXZ Bethe equations and leads to the correspondence between the
two.  This algebraic structure is a special case of a more general
phenomenon,  which appears in the representation theory of quantum
affine algebras.  Hernandez and Jimbo have introduced a category $\cO$ of
highest weight
representations of a Borel subalgebra $U_q(\hat{\mathfrak{b}}_+)$ of the quantum affine algebra $U_q(\hat{\mathfrak{g}})$~\cite{Hernandez_2012}.
This category includes all finite-dimensional representations of
$U_q(\hat{\mathfrak{g}})$, but also contains infinite-dimensional
representations on which the Borel action does not extend to the
entire quantum affine algebra.  Within this category, they have
constructed certain ``prefundamental modules'' which generate the
Grothendieck ring of $\cO$.  The relations among these generators
give rise to so-called $QQ$-systems (or $Q\tilde{Q}$-systems), which
generalize the the quantum Wronskian relations for
$\hat{\mathfrak{sl_2}}$ obtained in \cite{Bazhanov:1998dq}.

These $QQ$-systems, in their manifestation as relations in the
Grothendieck ring, have arisen in various circumstances
\cite{Krichever_1997, Mukhin_2003, Kazakov_2016, Bittmann_2020}.
However, they also have an explicit realization in terms of certain
polynomial equations.  As this is the form in which
$QQ$-systems appear in the context of $q$-opers, we briefly explain
how this works.

Upon fixing a finite-dimensional representation $\mathcal{H}$ of
$U_q(\hat{\mathfrak{g}})$ (known as the physical space), one obtains a
particular XXZ integrable model associated to the given quantum group.
The \emph{spectrum} of this integrable model is described by the
eigenvalues of the \emph{transfer matrices}, which are constructed in
the following way.  The $R$-matrix, associated to
$U_q(\hat{\mathfrak{g}})$ belongs to
$U_q(\hat{\mathfrak{b}}_+)\otimes U_q(\hat{\mathfrak{b}}_-)$, where
$\mathfrak{b}_\pm$ are opposite Borel subalgebras. For any
$\cV\in\cO$, the $R$-matrix acts on $\cV\otimes\cH$, and one can now
take the weighted trace of the $R$-matrix with respect to the first
factor. This trace is an operator acting on $\cH$.  In the particular
cases where $\cV$ is a finite-dimensional representation of $U_q(\hat{\mathfrak{g}})$
or a prefundamental module, the resulting operators are
called transfer matrices and Baxter $Q$-operators respectively.  Upon
certain canonical rescalings of the transfer matrices and
$Q$-operators, the eigenvalues of these operators become
polynomials, and the relations in the Grothendieck ring turn into
polynomial equations~\cite{Frenkel:2013uda,Frenkel_2018},
thereby proving conjectures on transfer matrices found in
\cite{Frenkel:ls}.

We observe that the $QQ$-system which describes twisted $(\SL(N),q)$-opers is the one associated to 
$\hat{\mathfrak{sl}}_n$ in \cite{Frenkel_2018}.

\subsection{Quantum/classical duality and applications to enumerative geometry} 

Quantum/classical duality is a relationship between a quantum and a
classical integrable system.  Well-known examples are the relationship
in type $A$ between the Gaudin model and the rational Calogero-Moser
system and between the XXX spin-chain and the rational
Ruijsenaars-Schneider model.  Both of these can be viewed as limits of
the duality between the XXZ spin-chain and the trigonometric
Ruijsenaars-Schneider model
\cite{Hal:2015,2012JMP....53l3512H,Mukhin:2009}.\footnote{We refer the
  reader to Section 4 of \cite{Gaiotto:2013bwa} for more information
  on quantum/classical duality and additional references.} This
duality is given by a transformation relating two sets of generators
for the quantum $K$-theory ring of cotangent bundles of full flag
varieties \cite{Koroteev:2017aa}.  One set of generators is obtained
from the XXZ Bethe equations.  One considers certain Bethe equations
where the dominant weights all come from the defining representation
and then takes symmetric functions on the corresponding Bethe roots.
The other generators are functions on a certain Lagrangian subvariety
in the phase space for the tRS model.

This correspondence has a direct interpretation in terms of twisted
$q$-opers; indeed, it may be viewed as a special case of the
$q$-Langlands correspondence.  As we discussed in the previous
section, Bethe equations arise from nondegenerate twisted $q$-opers.
The Bethe roots are precisely those zeros of the quantum Wronskians
associated to the $q$-oper which are not singularities of the
underlying $q$-connection.  On the other hand, there is an embedding
of the tRS model into the space of twisted $q$-opers.  More precisely,
a $q$-oper structure on a given $q$-connection $(E,A)$ is determined
uniquely by a full flag $\cL_\bullet$ of vector subbundles which
behave in a specified way with respect to $A$.  A section $s$
generating the line bundle $\cL_1$ over $\P^1\setminus\infty$ may be
viewed as an $N$-tuple of monic polynomials $(s_1,\dots,s_N)$.  If
these polynomials are all linear, then their constant terms are
precisely the momenta in the phase space of the tRS model.
Quantum/classical duality is then equivalent to the statement that the
Bethe roots and the constant terms of these monic linear polynomial
both give coordinates for an appropriate spaces of twisted $q$-opers.

If the monic polynomials $s_i$ are no longer linear, it is still the
case that the Bethe roots and the coefficients of these polynomials
are equivalent sets of coordinates for a space of twisted $q$-opers.
It is more complicated to interpret this statement as a duality
between the XXZ spin-chain and a classical multiparticle integrable
system. However, we do get an application to the quantum $K$-theory of
the cotangent bundles of partial flag varieties. This $K$-theory ring
is again generated by symmetric functions in appropriate Bethe roots.
In \cite{Rimanyi:2014ef}, Rimanyi, Tarasov, and Varchenko gave another
conjectural set of generators for this ring.  We show that these
generators are precisely those obtained from the coordinates for the
set of twisted $q$-opers coming from the coefficients of the
polynomials $s_i$, thereby proving this conjecture.

We remark that other applications of the XXZ Bethe ansatz equations to
geometry have appeared in the recent physics literature~\cite{Nekrasov:2009ui,Nekrasov:2009uh,Gaiotto:2013bwa,Nekrasov:2013aa}.

\subsection{Structure of the paper}
In \Secref{Sec:SL(N)Op}, we recall the relationship between
monodromy-free $\SL(N)$-opers with regular singularities on the
projective line and Gaudin models \cite{Frenkel:aa, Frenkel:ab}. We
follow an approach hinted at in \cite{Gaiotto:2011nm}, describing
opers in terms of vector bundles instead of principal bundles and
obtaining the Bethe equations from Wronskian relations.  We also
discuss the correspondence between an inhomogeneous version of the
Gaudin model and opers with an irregular singularity at infinity.

Next, in \Secref{Sec:SL2qOps}, we consider a $q$-deformation of opers
in the case of $\SL(2)$.  We adapt the techniques of the previous
section to give a correspondence between twisted $q$-opers and the
Bethe ansatz equations for the XXZ spin chain for $\sl_2$.
In \Secref{Sec:SLNqOps}, we generalize these constructions to
$\SL(N)$ and again prove a correspondence between $q$-opers and the
XXZ spin chain model.  We then discuss the case of $\SL(3)$ in detail
in \Secref{Sec:SL3qOps}.

In \Secref{Sec:Limits}, we consider classical limits of our results.
We show that an appropriate limit leads to a correspondence between a
twisted analogue of the discrete opers considered in
\cite{Mukhin_2005} and the spectra of a version of the XXX spin
chain.  By taking a further limit, we recover the relationship between
opers with an irregular singularity and the inhomogeneous Gaudin
model \cite{FFTL:2010,Rybnikov:2010}.

Finally, \Secref{Sec:Kth} is devoted to some geometric implications of
the results of this paper.  The quantum $K$-theory ring of the
cotangent bundle to the variety of partial flags is known to be
described via the Bethe ansatz equations
\cite{Koroteev:2017aa}.  We find a new set of generators defined in
terms of canonical coordinates on an appropriate set of $q$-opers.
These generators turn out to be the same as the conjectural generators
given in \cite{Rimanyi:2014ef}.

\addtocontents{toc}{\protect\setcounter{tocdepth}{1}}
\subsection*{Acknowledgments}
P.K. and A.M.Z. are grateful to the 2018 Simons Summer Workshop for
providing a wonderful working atmosphere in the early stages of this
project. P.K. thanks the organizers of the program ``Exactly Solvable
Models of Quantum Field Theory and Statistical Mechanics'' at the
Simons Center for Geometry and Physics, where part of this work was
done. We are also indebted to N. Nekrasov, A. Schwarz, and Y.
Soibelman for stimulating discussions and suggestions.  D.S.S. was
partially supported by NSF grant DMS-1503555, and A.M.Z. was partially
supported by a Simons Collaboration grant. P.K. is partly supported by AMS-Simons grant.

\addtocontents{toc}{\protect\setcounter{tocdepth}{2}}

\section{$\SL(N)$-opers with trivial monodromy and regular singularities}\label{Sec:SL(N)Op}

\subsection{$\SL(2)$-opers and Bethe equations}
In this section, we describe a simple reformulation of the results of
\cite{Frenkel:aa, Frenkel:ab} due to Gaiotto and Witten
\cite{Gaiotto:2011nm}.
\begin{Def} A \emph{$\GL(2)$-oper} on $\mathbb{P}^1$ is a triple
  $(E,\nabla,\cL)$, where $E$ is a rank $2$ vector
  bundle on $\mathbb{P}^1$, $\nabla:E\to E\otimes K$ is a connection
  (here $K$ is the canonical bundle), and $\cL$ is a line subbundle
  such that the induced map $\bar{\nabla}:\mathcal{L}\to
  E/\mathcal{L}\otimes K$ is an isomorphism. The triple is called an
  \emph{$\SL(2)$-oper} if the structure group of the flat
  $\GL(2)$-bundle may be reduced to $\SL(2)$.
\end{Def}
We always assume that the vector bundle $E$ is trivializable.

The oper condition may be checked explicitly in terms of a determinant
condition on local sections.  Indeed, $\bar{\nabla}$ is an isomorphism
in a neighborhood of a given point $z$ if for some (or for any)  local
section $s$ of $\cL$ with $s(z)\ne 0$, 
$$s(z)\wedge \nabla_z s(z)\neq 0.$$ Here,
$\nabla_z=\iota_{\frac{d}{dz}}\circ \nabla$, where
$\iota_{\frac{d}{dz}}$ is the inner derivation by the vector field $\frac{d}{dz}$.

In this section, we will be interested in $\SL(2)$-opers with regular
singularities.  An $\SL(2)$-oper with regular singularities of weights
$k_1,\dots,k_L,k_\infty$ at the points $z_1,\dots,z_L,\infty$ is a
triple $(E,\nabla,\cL)$ as above where $\bar{\nabla}$ is an
isomorphism everywhere except at each $z_i$ (resp. $\infty$), where it
has a zero of order $k_i$ (resp. $k_\infty$).  Concretely, near the
point $z_i$, we have
\begin{equation}\label{deg}
s(z)\wedge \nabla_z s(z)\sim (z-z_i)^{k_i}.
\end{equation}
We will always assume that our opers have trivial monodromy, i.e.,
that the monodromy of the connection around each $z_i$ is
trivial. This means that after an appropriate gauge change, we can
assume that the connection is trivial. (Recall that changing the
trivialization of $E$ by $g(z)$ induces gauge change in the
connection; explicitly, the new connection is
$g(z)\nabla g(z)^{-1}-d(g(z))g(z)^{-1}$.)
In terms of this
trivialization of $E$ over $\mathbb{P}^1\setminus\infty$, the line
bundle $\cL$ is generated over this affine space by the section
\begin{equation}
s=\begin{pmatrix}   {q}_{+}(z) \\
   {q}_-(z)
 \end{pmatrix},
\end{equation}
where $q_{\pm}(z)$ are polynomials without common roots.
The condition \eqref{deg} leads to the following equation on the \emph{Wronskian}:
\begin{equation}\label{Wron}
q_+(z)\partial_zq_{-}(z)-\partial_zq_+(z)q_-(z)=\rho(z),
\end{equation}
where $\rho(z)$ is a polynomial whose zeros are determined by
\eqref{deg}.  After multiplying $s$ by a constant, we may take
$\rho(z)=\prod^L_{i=1}(z-z_i)^{k_i}$.  By applying a constant gauge
transformation in $\SL(2,\C)$, we may further normalize $s$ so that
$\deg(q_{-})<\deg(q_{+})$ and $q_{-}(z)=\prod^{l_-}_{i=1}(z-w_i)$ has
leading coefficient $1$.  (More precisely, transforming by
$\left(\begin{smallmatrix} 0 & 1\\ -1 & 0
  \end{smallmatrix}\right)$ if necessary allows us to assume that
$\deg(q_{-})\le\deg(q_{+})$; if the degrees are equal, transforming by
an elementary matrix brings us to the case $\deg(q_{-})<\deg(q_{+})$.
The final reduction uses a diagonal gauge change.)

We now make the further assumption that our oper is
\emph{nondegenerate}, meaning that none of the $z_i$'s are roots of
$q_{-}$.  It is now an immediate consequence of \eqref{Wron} that each
root of $q_{-}$ has multiplicity $1$.

Let $k=\sum^L_{i=1} k_i$ denote $\deg(\rho)$.  An easy calculation
using the fact that $\deg(q_-)<\deg(q_+)$ gives
$\deg(q_-)+\deg(q_+)=k+1$; this implies that $\deg(q_{-})=l_-\leq
k/2$.   We now rewrite \eqref{Wron} in the equivalent form
\begin{equation}\label{quot}
\partial_z\left(\frac{q_{+}(z)}{q_{-}(z)}\right)=-\frac{\rho(z)}{q_{-}(z)^2}.
\end{equation}
Since the residue at each $w_i$ of the left-hand side of
this equation is $0$, computing the residues of the right-hand side
leads to the conditions
\begin{equation}\label{sl2bethe}
\sum_m\frac{k_m}{z_m-w_i}=\sum_{j\neq i}\frac{2}{w_j-w_i}, \qquad i=1,\dots, l_-.
\end{equation}
These are the Bethe ansatz equations for the $\sl_2$-Gaudin
model at level $k-2l_{-}\ge 0$; they determine the spectrum of this
model.

A local section for $\cL$ at $\infty$ is given by 
\begin{equation}
\begin{pmatrix}   \tilde{q}_{+}(\tilde z) \\
   \tilde{q}_-(\tilde z)
 \end{pmatrix}
 =\tilde{z}^{l_{+}}\begin{pmatrix}
   q_+(1/\tilde{z}) \\
  q_-(1/\tilde{z})  
\end{pmatrix},
\end{equation}
where $l_{+}=\deg(q_{+})$.  If we set $k_{\infty}=k-2l_{-}=l_{+}-l_{-}-1$, we obtain
\begin{equation}
 \tilde{q}_+(\tilde z)\partial_{\tz}  \tilde{q}_{-}(\tilde z)-\partial_{\tz} \tilde{q}_+(\tilde z) \tilde{q}_-(\tilde z)\sim  \tz^{k_{\infty}}.
\end{equation}
Thus, we have proved  the following theorem. 
\begin{Thm}\label{operthm}
  There is a one-to-one correspondence between the spectrum of the
  Gaudin model, described by the Bethe equations for
  dominant weights, and the space of nondegenerate $\SL(2)$-opers with trivial
  monodromy and regular singularities at the points $z_1,
  \dots, z_L, \infty$ with weights $k_1, \dots, k_L, k_{\infty}$.
\end{Thm}

\subsection{Miura opers and the Miura transformation}

The previous theorem raises the natural question of whether one can
give a geometric interpretation to solutions of the Bethe equations
without assuming that the level $k-2l_-$ is nonnegative.  \emph{Miura
  opers} provide such an description.  A Miura oper is an oper
$(E,\nabla,\cL)$ together with an additional line bundle $\hcL$
preserved by $\nabla$.  There may be a finite set of points where
$\cL$ and $\hcL$ do not span $E$.  It turns out that one can associate
to any oper with regular singularities a family of Miura opers
parameterized by the flag variety~\cite{Frenkel:aa}.

Given a Miura oper, we may choose a trivialization of $E$ so that the
line bundle $\hcL$ is generated by the section $\hat{s}=(1,0)$.  We retain
our notation for the section
$s=\left(\begin{smallmatrix}q_{+}\\q_{-}
  \end{smallmatrix}
\right)$ generating $\cL$, but here, we do not impose any restrictions
on $\deg(q_{-})$.

Theorem~\ref{operthm} can be generalized to give the following
theorem, which is proved in a similar way.

\begin{Thm}\label{Miurasl2}
  There is a one-to-one correspondence between the set of solutions of
  the Bethe Ansatz equations \eqref{sl2bethe} and the set of
  nondegenerate $\SL(2)$-Miura opers with trivial monodromy and regular
  singularities at the points $z_1, \dots, z_L, \infty$ with weights
  at the finite points given by $k_1, \dots, k_L$.
\end{Thm}

We now give a different formulation of $\SL(2)$-opers which shows how
the eigenvalues of the Gaudin Hamiltonian can be seen directly from
the oper. We will do this by applying several $\SL(2)$-gauge transformations to our
trivial connection to reduce it to a canonical form.  We start with
a gauge change by $g(z)=\left(\begin{smallmatrix}  q_-(z) & -q_+(z)\\
   0 &  q_-^{-1}(z)
 \end{smallmatrix}\right)$; note that  $g(z)s(z)=\left(\begin{smallmatrix}0\\1
  \end{smallmatrix}\right)$.  The new connection matrix is 
\begin{equation} -(\partial_z g)g^{-1}= -\begin{pmatrix}\partial_z
    q_-(z)& -\partial_z q_+\\
   0 &   -\frac{\partial_z q_-(z)}{q_-(z)^2}
 \end{pmatrix}
 \begin{pmatrix}q_{-}^{-1}(z) & q_+(z)\\
   0 &   q_{-}(z)
 \end{pmatrix}=\begin{pmatrix}\frac{-\partial_z q_-(z)}{q_-(z)}& -\rho(z)\\
   0 &   \frac{\partial_z q_-(z)}{q_-(z)}
 \end{pmatrix}.
\end{equation}
  Next, the diagonal transformation $\left(\begin{smallmatrix} \rho(z)^{-1/2}& 0\\
   0 &   \rho(z)^{1/2}
 \end{smallmatrix}\right)$
  brings us to the \emph{Cartan connection}
  \begin{equation}
  A(z)=\begin{pmatrix}
    -u(z)& -1\\
   0 &   u(z)
 \end{pmatrix},
\end{equation}

  where $$u(z)=-\frac{\partial_z\rho(z)}{2\rho(z)}+\frac{\partial_z
    q_-(z)}{q_-(z)}=-\sum_m\frac{k_m/2}{z-z_m}+\sum_i
  \frac{1}{z-w_i}.$$
  Finally, we apply the \emph{Miura transformation}: gauge change by the lower triangular matrix  $\left(\begin{smallmatrix} 1& 0\\
      u(z) & 1
 \end{smallmatrix}\right)$ gives the connection matrix 
 \begin{equation}\label{e:miuraform}
 B(z)=\begin{pmatrix}
    0& -1\\
  - t(z) &   0
\end{pmatrix},\qquad\text{ where $\,t(z)=\partial_z u(z)+u^2(z)$.}
\end{equation}

An explicit computation using the Bethe equations \eqref{sl2bethe}
gives 
\begin{equation*}
t(z)=\sum_m\frac{k_m(k_m+2)/4}{ (z-z_m)^2}+\sum_m{\frac{c_m}{z-z_m}},
\end{equation*}
where
\begin{equation*}
c_m=k_m\bigg(\sum_{n\neq m}\frac{k_n/2}{z_m-z_n}-\sum_{i=1}^{l_-}\frac{1}{z_m-w_i}\bigg).
\end{equation*}
This shows that $t(z)$ does not have any singularities at $z= w_i$;
moreover, since the $c_m$ are the eigenvalues of the Gaudin
Hamiltonians, it depends only on this spectrum. In particular, the
Gaudin eigenvalues can be read off explicitly as the negative of the
residue of the nonconstant entry of the
connection matrices $B(z)$.  Note that a horizontal section
$f=\left(\begin{smallmatrix}f_1\\f_2
  \end{smallmatrix}\right)$ to the connection in this gauge  is
determined by a solution to the linear differential equation \begin{equation}
(\partial_z^2-t(z))f_1(z)=0.
\end{equation} 
The differential operator $\partial_z^2-t(z)$ can be viewed as a
projective connection.

\subsection{Generalization to $\SL(N)$: a brief summary}\label{s:slnclassical}

We now give a brief description of the interpretation of the spectrum
of the $\sl_N$-Gaudin model in terms of $\SL(N)$-opers.

\begin{Def} A \emph{$\GL(N)$-oper} on $\mathbb{P}^1$ is a triple
  $(E,\nabla,\cL_\bullet)$, where $E$ is a rank $n$ vector bundle on
  $\mathbb{P}^1$, $\nabla:E\to E\otimes K$ is a connection, and
  $\cL_\bullet$ is a complete flag of subbundles such that $\nabla$
  maps $\cL_i$ into $\cL_{i+1}\otimes K$ and the induced maps
  $\bar{\nabla}_i:\cL_i/\cL_{i-1}\to \cL_{i+1}/\cL_i\otimes K$ are
  isomorphisms for $i=1,\dots,N-1$. The triple is called an
  \emph{$\SL(N)$-oper} if the structure group of the flat
  $\GL(N)$-bundle may be reduced to $\SL(N)$.
\end{Def}
As in the $\SL(2)$-case, one can interpret the fact that the $\bar{\nabla}_i$'s are
isomorphisms in terms of the nonvanishing of certain determinants
involving local sections of $\cL_1$.  Given a local section $s$ of
$\cL_1$, for $i=1,\dots,N$, let
\begin{equation*} \cW_i(s)(z)=\left.\left(s(z)\wedge\nabla_z
    s(z)\wedge\dots\wedge\nabla_z^{i-1} s(z)\right)\right|_{\Lambda^i\cL_{i}}
\end{equation*}
Then $(E,\nabla,\cL_\bullet)$ is an oper if and only if for each $z$,
there exists a local section of $\cL_1$ for which $\cW_i(s)(z)\ne 0$
for all $i$.  Note that $\cW_1(s)\ne 0$ simply means that $s$ locally
generates $\cL_1$.

We again will need to relax the isomorphism condition in the above
definition to allow the oper to have regular singularities.  Recall
that the weight lattice for $\SL(N)$ is the free abelian group on the
fundamental  weights $\om_1,\dots,\om_{N-1}$.  Moreover, a weight is
dominant if it is a nonnegative linear combination of the $\om_i$'s.

Fix a collection of points $z_1,\dots,z_L$ and corresponding dominant
integral weights $\lambda_1,\dots,\lambda_L$.  Write $\lambda_m=\sum
l_m^i\om_i$. We say that $(E,\nabla,\cL_\bullet)$ is an $\SL(N)$-oper
with regular singularities of weights $\lambda_1,\dots,\lambda_L$ at
$z_1,\dots,z_L$ if $(E,\nabla)$ is a flat $\SL(N)$-bundle, and each of
the $\bar{\nabla}_i$'s is an isomorphism except possibly at $z_m$,
where it has a zero of order $l_m^{i}$, and $\infty$.  The conditions at the
singularities may be expressed equivalently in terms of a nonvanishing
local section. For each $j$ with $1\le j \le N-1$, set $\Lambda_{j} =
\prod^{L}_{m=1}(z-z_m)^{l_m^{j}}$ and $\ell_m^{j}=\sum_{k=1}^{j}
l_m^{k}$.  Then, for $2\le i\le N$,
\begin{equation}
\cW_i(s)(z)\sim P_{i-1}:= \Lambda_1(z) \Lambda_2(z)\cdots\Lambda_{i-1} (z)= \prod^{L}_{m=1}(z-z_m)^{\ell_m^{i-1}}.
\end{equation}

As we saw for $\SL(2)$, to get the Bethe equations for nondominant
weights, we need to introduce Miura opers.  Again, a Miura oper is a
quadruple $(E,\nabla,\cL_\bullet,\hcL_\bullet)$ where
$(E,\nabla,\cL_\bullet)$ is an oper with regular singularities and
$\hcL_\bullet$ is a complete flag of subbundles preserved by $\nabla$.
Given a Miura oper, choose a trivialization of $E$ on
$\mathbb{P}^1\setminus\infty$ such that $\hcL_\bullet$ is the standard
flag, i.e., the flag generated by the ordered basis $e_1,\dots, e_N$.
If $s$ is a section generating $\cL_1$ on this affine line, consider
the following determinants for $i=1,\dots,N$:
\begin{equation*} \cD_i(s)(z)=e_1\wedge\dots\wedge e_{N-i}\wedge
  s(z)\wedge\nabla_z s(z)\wedge \dots\wedge\nabla_z^{i-1} s(z).
\end{equation*}
Each of these is a polynomial multiple of the volume form.  Note that
$\cD_N(s)(z)=\cW_N(s)(z)$; in particular, $\cD_N(s)(z)\ne 0$ away from
the $z_m$'s.  We will call a Miura oper \emph{nondegenerate} if the
orders of the zero of $\cD_i(s)$ and $\mathcal{W}_i(s)$ at each $z_m$
are the same and moreover, if $\cD_i(s)$ and $\cD_k(s)$ for $i\ne k$
both vanish at a point $z$, then $z=z_m$ for some $m$.

These conditions may be expressed in a more Lie-theoretic
form. Let $B$ be the upper triangular Borel subgroup of $\SL(N)$.
Under the usual identification of $\SL(N)/B$ as the variety of
complete flags, $B$ corresponds to the standard flag $\cE$.  If $\cF$
is another flag, we say that $(\cE,\cF)$ have relative position $w$
(with $w$ an element of the Weyl group $S_N$) if $\cF=g\cdot\cE$ for
some $g$ in the double coset $BwB$.  If the relative position is
$w_0$, where $w_0$ is the longest element given by the permutation
$i\mapsto N+1-i$ for all $i$, we say that the flags are in general position.

Given an ordered basis $f=(f_1,\dots,f_N)$ for $\mathbb{C}^N$, let
$Q_k(f)=e_1\wedge\dots\wedge e_{N-k}\wedge f_1\wedge\dots\wedge f_k$.
It is immediate that the zeros of the function $k\mapsto Q_k(f)$
depend only on the flag determined by $f$.  (Of course, $Q_N(f)$ is
always nonzero, since $f$ is a basis.)  Let $\sigma_k=(k\ k+1)\in
S_N$.
\begin{Lem} Let $\cF$ be a flag determined by the ordered basis
  $f=(f_1,\dots,f_N)$. 
\begin{enumerate}\item The pair $(\cE,\cF)$ are in general position if and only if $Q_j(f)\ne 0$ for all $j$.
\item The pair $(\cE,\cF)$ have relative position $w_0 \sigma_k$ if and
  only if $Q_k(f)= 0$ and $Q_j(f)\ne 0$ for all $j\ne k$.
\end{enumerate}
\end{Lem}
\begin{proof}  In both cases, the forward implication is an easy direct
  calculation.  For example, if $(\cE,\cF)$ are in general position,
  then $\cF=bw_0\cE$ for some $b\in B$, hence is determined by the
  ordered basis $f_i=bw_0e_i=be_{N+1-i}$.  Since $Q_j(f)\ne 0$ is
  equivalent to the fact that the projection of $\Span(f_1,\dots,f_j)$
  onto $\Span(e_{N+1-j},\dots,e_N)$ is an isomorphism, it is now
  immediate that for $(\cE,\cF)$ in general position, $Q_j(f)\ne 0$
  for all $j$.  

For the converse, first assume that $Q_j(f)\ne 0$ for
  all $j$.  One shows inductively that the basis $f$ can be
  modified to give a new ordered basis $\hf$ for $\cF$ for which the matrix
  $b=\left(\begin{smallmatrix} \hf_N & \hf_{N-1}\dots & \hf_1
    \end{smallmatrix}\right)\in B$.  Thus, $\cF=b w_0\cE$.

Now, assume that $Q_k(f)=0$, but the other $Q_j(f)$'s are nonzero. The
same argument as above shows that without loss of generality, we may
assume that for  $j=1,\dots,k-1$, $f_j$ is a column vector with lowest nonzero component in the
$N-j$ place.  We may further assume that all other $f_i$'s have bottom
$k-1$ components zero.  Since $Q_k(f)=0$, $(f_k)_{N-k}=0$.  However,
$Q_{k+1}(f)\ne 0$ now gives $(f_{k+1})_{N-k}\ne 0$ and
$(f_k)_{N-k-1}\ne 0$.  It is now clear that the flag $\cF$ is
determined by an ordered basis $\hf$ for which 
$b=\left(\begin{smallmatrix} \hf_N & \dots \hf_{k} & \hf_{k+1}\dots & \hf_1
    \end{smallmatrix}\right)\in B$.  This means that $\cF= b w_0 w_k\cE$.

\end{proof}

Returning to our Miura oper, recall that $s(z),\nabla_z
s(z),\dots,\nabla_z^{N-1} s(z)$ is an ordered basis for the flag
$\cL(z)$ as long as $z$ is not a singular point.  If we denote this
basis by $\bm{s}(z)$, we see that
$\cD_i(s)(z)=Q_i(\bm{s}(z))$.  The lemma now shows that the fact
that the $\cD_i(s)$'s have no roots in common outside of regular
singularities is equivalent to the statement that the relative
position of $(\hcL_\bullet(z),\cL_\bullet(z))$ is either $w_0$ or $w_0
\sigma_k$ for some $k$.  Furthermore, $s(z),(z-z_m)^{-l_m^1}\nabla_z
s(z),\dots,(z-z_m)^{-l_m^{N-1}}\nabla_z^{N-1}s(z)$ is an ordered basis
for $\cL_\bullet$ at $z_m$. Hence, $\cD_i(s)(z)$ and $\cW_i(s)(z)$
having zeros of the same order at $z_m$ is equivalent to the fact that
the flags $\hcL_\bullet(z_m)$ and $\cL_\bullet(z_m)$ are in general
position.

The determinant conditions for the zeros of $\mathcal{D}_k(s)$ lead
to Bethe equations in the same way as before \cite{Frenkel:aa}: 
\begin{equation}\label{Betheeqsln}
\sum^L_{i=1}\frac{\langle \lambda_i, \check{\alpha}_{i_j}\rangle}{w_j-z_i}=\sum _{s\neq j}
\frac{\langle \check{\alpha}_{i_s}, \check{\alpha}_{i_j}\rangle}{w_s-w_j}
\end{equation}
where the $w_j$'s are distinct points corresponding to zeros of the
determinants $\mathcal{D_k}(s)$. 

We can now state the $\SL(N)$ analogue of Theorem~\ref{Miurasl2}.
Here, $\lambda_\infty$ is a dominant weight determined by the
$\lambda_i$'s and the $\alpha_{i_j}$'s.

\begin{Thm} There is a one-to-one correspondence between the set of
  solutions to the Bethe ansatz equations \eqref{Betheeqsln} and the
  set of nondegenerate $\SL(N)$-Miura opers with trivial monodromy and
  regular singularities at the points $z_1, \dots, z_L,\infty$ with
  weights $\lambda_1, \dots, \lambda_L,\lambda_\infty$.
\end{Thm}

\subsection{Irregular singularities}\label{sec:irrsing}

In this section, we recall the relationship between opers with
irregular as well as regular singularities and an inhomogeneous
version of the Gaudin model introduced in
\cite{FFTL:2010,Rybnikov:2010}.  Here, we will only consider the
simplest case of a double pole irregularity at $\infty$.  We also
restrict the discussion to $\SL(2)$.

Let $(E,\nabla,\cL)$ be an $\SL(2)$-oper with regular singularities on
$\P^1\setminus\infty$ whose underlying connection is gauge equivalent
to $d+\bfa\,dz$, where $\bfa=\diag(a,-a)$ with $a\ne 0$.  Changing
variables to $1/z$, we see that this connection has a double pole at
$\infty$.  It is no longer possible to trivialize the connection
algebraically, but it can be trivialized using the exponential
transformation $h(z)=e^{\bfa z}$.  If we let
$\left(\begin{smallmatrix} q_{+}(z)\\q_{-}(z)
  \end{smallmatrix}\right)$ be a section generating the line bundle
$\cL$  (so  $q_{+}(z)$ and $q_{-}(z)$ are polynomials with no common
zeros), then in the trivial gauge, this section becomes
  \begin{equation*}s(z)=e^{-{\bfa}z}\begin{pmatrix}
   q_+(z)  \\
   q_-(z)  
 \end{pmatrix}.
\end{equation*}
Note that we cannot assume that
$\deg(q_{-})<\deg(q_{+})$, since the necessary constant gauge changes
do not preserve $d+\bfa\,dz$.  However, we can assume that $q_{-}$ is
monic: $q_{-}(z)=\prod^{l_{-}}_{i=1}(z-w_i)$. 

The condition $s(z)\wedge \nabla_z s(z)=\rho(z)$ gives  a ``twisted''
form of the Wronskian:
\begin{equation}\label{twron}
  q_{+}(z)\partial_zq_-(z)-   q_{-}(z)\partial_zq_+(z)+2aq_{+}(z)q_{-}(z)=\rho(z)
\end{equation}
As before, we assume this oper is nondegenerate, i.e., $q_-(z_m)\ne 0$
for all $m$; again, this implies that the zeros of $q_-$ are simple.

To compute the Bethe ansatz equations, we observe that after
multiplying \eqref{twron} by $-e^{-2a z}/(q_-(z))^2$, we obtain 
\begin{equation} \partial_z\left(-e^{-2a z}\frac{q_+(z)}{q_-(z)}\right)=\frac{e^{-2a z}\rho(z)}{q_-(z)^2}.
\end{equation}
Taking residues at each $w_i$ now leads to the inhomogeneous Bethe equations
\begin{equation}\label{gbetheir}
-2a+\sum_m\frac{k_n}{z_n-w_i}=\sum_{j\neq i}\frac{2}{w_j-w_i}, \qquad i=1,\dots, l_-.
\end{equation}

We thus obtain the following theorem:
\begin{Thm}
  There is a one-to-one correspondence between the set of solutions of
  the inhomogeneous Bethe equations \eqref{gbetheir} and the set of
  nondegenerate $\SL(2)$-opers with regular singularities at the
  points $z_1, \dots, z_L$ of weights $k_1, \dots, k_L$ at the points
  $z_1, \dots, z_L$ and with a double pole with $2$-residue $-\bfa$.
\end{Thm}
Although we will not state it explicitly, there is a similar result
for $\SL(N)$.  Applying the methods of Section~\ref{s:slnclassical},
one shows that twisted Wronskians arise from nondegenerate
$\SL(N)$-opers with regular singularities at fixed finite points and a
double pole with regular semisimple $2$-residue at infinity.  One now
obtains solutions of the Bethe equations from these twisted
Wronskians.

We remark that for the opers considered in this section, there is no
longer an entire flag variety of associated Miura opers.  Indeed, the
only line bundles $\hcL$ preserved by $d+\bfa\, dz$ are those
generated by $e_1$ and $e_2$.  More generally, consider an
$\SL(N)$-oper with underlying connection $d+A\,dz$, where $A$ is a
diagonal matrix with distinct eigenvalues. The
flags $\hcL_\bullet$ preserved by this connection are precisely those
generated by ordered bases obtained by permuting the standard basis.
Hence, the associated Miura opers are parameterized by the Weyl group.

\section{$(\SL(2),q)$-opers}\label{Sec:SL2qOps}
\subsection{Definitions}

We now consider a $q$-deformation of the set-up in the previous
section.  It involves a difference equation version of connections and
opers.

Fix $q\in\C^*$.  Given a vector bundle $E$ over $\P^1$, let $E^q$
denote the pullback of $E$ under the map $z\mapsto qz$.  We will always
assume that $E$ is trivializable.  Consider a
map of vector bundles $A:E\to E^q$.  
Upon picking a trivialization, the map $A$ is determined by a matrix
$A(z)\in\gl(N,\C(z))$ giving the linear map $E_z\to E_{qz}$ in the given
bases.  A change in trivialization by $g(z)$ changes the
matrix via 
\begin{equation}\label{qgauge}A(z)\mapsto g(qz)A(z)g^{-1}(z);
\end{equation}
thus, $q$-gauge change is twisted conjugation. Let $D_q:E\to E^q$ be
the operator that takes a section $s(z)$ to $s(qz)$.  We associate the
map $A$ to the difference equation
$D_q(s)=As$.

\begin{Def} A meromorphic \emph{$(\GL(N),q)$-connection} over $\P^1$
  is a pair $(E,A)$, where $E$ is a (trivializable) vector bundle of
  rank $N$ over $\P^1$ and $A$ is a meromorphic section of the sheaf
  $\Hom_{\cO_{\P^1}}(E,E^q)$ for which $A(z)$ is invertible, i.e. lies
  in $\GL(N,\C(z))$.  The pair
  $(E,A)$ is called an $(\SL(N),q)$-connection if there exists a
  trivialization for which $A(z)$ has determinant $1$.
\end{Def}
For simplicity, we will usually omit the word `meromorphic' when
referring to $q$-connections.

\begin{Rem} More generally, if $G$ is a complex reductive group, one
  can define a meromorphic $(G,q)$-connection over $\P^1$ as a pair
  $(\cG,A)$ where $\cG$ is a principal $G$-bundle over $\P^1$ and $A$
  is a meromorphic section of $\Hom_{\cO_{\P^1}}(\cG,\cG^q)$.
\end{Rem}

Next, we define a $q$-analogue of opers.  In this section, we will
restrict to type $A_1$.

\begin{Def} A \emph{$(\GL(2),q)$-oper} on $\P^1$ is a triple
  $(E,A,\cL)$, where $(E,A)$ is a $(\GL(2),q)$-connection and $\cL$ is
  a line subbundle such that the induced map $\bar{A}:\cL\to (E/\cL)^q$ is
  an isomorphism.  The triple is called an \emph{$(\SL(2),q)$-oper} if
  $(E,A)$ is an $(\SL(2),q)$-connection.
\end{Def}

The condition that $\bar{A}$ is an isomorphism can be made explicit in
terms of sections.  Indeed, it is equivalent to $$s(qz)\wedge A(z)
s(z)\neq 0$$ for $s(z)$ any section generating $\cL$ over either of
the standard affine coordinate charts.

From now on, we assume that $q$ is not a root of unity.  We want to
define a $q$-analogue of the opers considered in
Section~\ref{sec:irrsing}.  First, we introduce the notion of a
$q$-oper with regular singularities.  Let $z_1,\dots,z_L\ne 0,\infty$
be a collection of points such that $q^\Z z_m\cap q^\Z z_n=\varnothing$
for all $m\ne n$.
\begin{Def}
A \emph{$(\SL(2),q)$-oper with regular singularities} at the points $z_1,\dots,
z_L\neq 0, \infty$ with weights $k_1, \dots k_L$ is a meromorphic
$(\SL(2),q)$-oper $(E,A,\cL)$ for which $\bar{A}$
is an isomorphism  everywhere on $\P^1\setminus\{0,\infty\}$ except at the points $z_m$, $q^{-1}z_m$,
$q^{-2}z_m$, \dots , $q^{-k_m+1}z_m$ for $m\in \{1, \dots, L\}$, 
where it has simple zeros.
\end{Def}

The second condition can be restated in terms of a section $s(z)$
generating $\cL$ over $\P^1\setminus \infty$: $s(qz)\wedge A(z) s(z)$
has simple zeros at $z_m$, $q^{-1}z_m$, $q^{-2}z_m$, \dots ,
$q^{-k_m+1}z_m$ for every $m\in \{1, \dots, L\}$ and has no other
finite zeros.

Next, we define twisted $q$-opers; these are $q$-analogues of the
opers with a double pole singularity considered in
Section~\ref{sec:irrsing}. Let $Z=\diag(\ze,\ze^{-1})$ be a diagonal
matrix with $\ze\ne\pm 1$.
\begin{Def} A $(\SL(2),q)$-oper $(E,A,\cL)$ with regular singularities
  is called a \emph{$Z$-twisted $q$-oper} if $A$ is gauge-equivalent
  to $Z^{-1}$.
\end{Def}

Finally, we will need the notion of a \emph{Miura $q$-oper}.  As in
the classical case, this is a quadruple $(E,A,\cL,\hcL)$ where
$(E,A,\cL)$ is a $q$-oper and $\hcL$ is a line bundle preserved by $A$.

For the rest of Section~\ref{Sec:SL2qOps}, $(E, A,\mathcal{L})$ will be a
$Z$-twisted $(\SL(2),q)$-connection with regular singularities at
$z_1,\dots, z_L\neq 0, \infty$ having (nonnegative) weights $k_1, \dots k_L$.

\subsection{The quantum Wronskian and the Bethe ansatz}

Choose a trivialization for which the $q$-connection matrix is
$Z^{-1}$.  Since $\cL$ is trivial on $\P^1\setminus\infty$, it is
generated by a section 
\begin{equation} s(z)=\begin{pmatrix} Q_+(z)\\Q_-(z)
  \end{pmatrix},
\end{equation}
where  $Q_+(z)$ and $Q_-(z)$ are polynomials without common roots.
The regular singularity condition on the $q$-oper becomes an explicit
equation for the \emph{quantum Wronskian}:
 
\begin{equation}\label{eq:qOpDefSL2}\zeta^{-1} Q_{+}(z)Q_{-}(qz)-\zeta
  Q_{+}(qz)Q_{-}(z)=\rho(z):=\prod^{L}_{m=1}\prod^{k_{m}-1}_{j=0}(z-q^{-j}z_m).
\end{equation}

We can assume that $\rho$ is monic, since we can multiply $s$ by a
nonzero constant.  We are also free to perform a constant diagonal
gauge transformation, since this leaves the $q$-connection matrix
unchanged.  Thus, we may assume that $Q_-$ is monic, say
$Q_-(z)=\prod_{i=1}^{l-} (z-w_i)$.

We now restrict attention to \emph{nondegenerate} $q$-opers.  This
means the $q^\Z$-lattices generated by the roots of $\rho$ and $Q_{-}$
do not overlap, i.e., $q^\Z z_m\cap q^\Z w_i=\varnothing$ for all $m$
and $i$.  Note that this condition implies that $w_j\ne qw_i$ for all
$i,j$;   if $w_j=q w_i$, then  
 \eqref{eq:qOpDefSL2} shows that $w_i$ would be a common
 zero of $\rho$ and $Q_{-}$.

Evaluating \eqref{eq:qOpDefSL2} at $q^{-1}z$ gives
$\rho(q^{-1}z)=\zeta^{-1}Q_{+}(q^{-1}z)Q_{-}(z)-\zeta
Q_{+}(z)Q_{-}(q^{-1}z)$.  If we divide \eqref{eq:qOpDefSL2} by this
equation and evaluate at the zeros of $Q_-$, we obtain the following
constraints:
\begin{equation}
\frac{\rho(w_i)}{\rho(q^{-1}w_i)}=-\zeta^{-2}\frac{Q_-(qw_i)}{Q_-(q^{-1}w_i)}, 
\end{equation}
or more explicitly, setting $k=\sum{k_m}$, 
\begin{equation} q^{k}\prod_{m=1}^L\frac{w_i-q^{1-k_m}z_m}{w_i-q z_m}=-\zeta^{-2}\prod^{l_-}_{j=1}\frac{q w_i-w_j}{q^{-1}w_i-w_j}.
\end{equation}

Rewriting this equation, we obtain the
$\sl_2$ XXZ Bethe equations (see e.g. \cite{Reshetikhin:1987}):
\begin{equation}\label{sl2qbethe} \prod_{m=1}^L\frac{w_i-q^{1-k_m}z_m}{w_i-q
    z_m}=-\zeta^{-2}q^{l_{-} - k}\prod^{l_-}_{j=1}\frac{q
    w_i-w_j}{w_i-q w_j},\qquad i=1,\dots,l_{-}.
\end{equation}

We call a solution of the Bethe equations nondegenerate if the $q^\Z$
lattices generated by the $w_i$'s and $z_m$'s are disjoint for all $i$
and $m$.  We have proven the following theorem:

\begin{Thm} There is a one-to-one correspondence between the set of
  nondegenerate solutions of the $\sl_2$ XXZ Bethe equations \eqref{sl2qbethe} and
  the set of nondegenerate $Z$-twisted $(\SL(2),q)$-opers with regular singularities at the points $z_1,\dots, z_L\neq 0, \infty$ with weights $k_1, \dots k_L$.
\end{Thm}

\subsection{The $q$-Miura transformation and the transfer matrix}
We now consider the $q$-Miura transformation which puts the
$q$-connection matrix into a form analogous to \eqref{e:miuraform} in
the classical setting.  As we will see, the eigenvalue of the transfer
matrix for the XXZ model will appear explicitly in the $q$-connection matrix.

First, we consider the gauge change by \begin{equation}
g(z)=\begin{pmatrix}
   Q_-(z) & -Q_+(z)\\
   0 &  Q_-^{-1}(z)
 \end{pmatrix},
\end{equation} which takes the section
$s(z)$ into $g(z)s(z)=\left(\begin{smallmatrix}0\\1
  \end{smallmatrix}\right)$.
In this gauge, the $q$-connection matrix has the form
\begin{equation}\label{acon}\begin{aligned} A(z)&=\begin{pmatrix} Q_-(qz)\zeta^{-1} & -\zeta Q_+(qz)\\
   0 &  \zeta Q_-^{-1}(qz)
 \end{pmatrix}
\begin{pmatrix}  Q_-(z) & -Q_+(z)\\
   0 &  Q_-^{-1}(z)
 \end{pmatrix}^{-1}\\&=\begin{pmatrix} \zeta^{-1}Q_-(qz)Q^{-1}_-(z)&\rho(z)\\
   0 &  \zeta Q_-^{-1}(qz)Q_-(z)
 \end{pmatrix},
\end{aligned}
\end{equation}
where $\rho$ is the quantum Wronskian.

Before proceeding, we recall that every eigenvalue of the \emph{transfer matrix} for
the XXZ model has the form (see \cite{Reshetikhin:2010si}, Section 4 and references therein)
\begin{equation}
T(z)=\zeta^{-1}\rho(q^{-1}z)\frac{Q_-(qz)}{Q_-(z)}+\zeta \rho(z)\frac{Q_-(q^{-1}z)}{Q_-(z)}.
\end{equation}
For ease of notation, we set $a(z)=\zeta^{-1}Q_-(qz)Q^{-1}_-(z)$, so
that $A(z)=\left(\begin{smallmatrix}a(z)& \rho(z)\\
   0 &  a^{-1}(z)
 \end{smallmatrix}\right)$ and $T(z)=a(z)\rho(q^{-1}z)+a^{-1}(q^{-1}z)\rho(z)$.  We now apply the gauge transformation by
the matrix $\left(\begin{smallmatrix} 1 & 0\\ a(z)/\rho(z) & 1
  \end{smallmatrix}\right)$; this brings the $q$-connection into the
form 
\begin{equation}\label{mer}
\hat{A}(z)=\begin{pmatrix}
   0& \rho(z)\\
   -\rho^{-1}(z) &  T(qz)\rho^{-1}(qz)
 \end{pmatrix}.
\end{equation}
If $\left(\begin{smallmatrix} f_1\\ f_2  \end{smallmatrix}\right)$ is
a solution of the corresponding difference equation, then we have
$D_q(f_1)=\rho(z)f_2$ and $D_q(f_2)= -\rho^{-1}(z)f_1+
T(qz)\rho^{-1}(qz)f_2$.  Simplifying, we see that $f_1$ is a solution
of the second-order scalar difference equation
\begin{equation}\label{scalar} \left(D_q^2-T(qz)D_q-\frac{\rho(qz)}{\rho(z)}\right)f_1=0.
\end{equation}
Summing up, we have

\begin{Thm}
 Nondegenerate $Z$-twisted $(\SL(2),q)$-opers with regular
 singularities at the points $z_1,\dots, z_n\neq 0, \infty$ with
 weights $k_1, \dots k_n$ may be represented by meromorphic
 $q$-connections of the form \eqref{mer} or equivalently, by the
 second-order scalar difference operators \eqref{scalar}.
\end{Thm}

\subsection{Embedding of the tRS model into $q$-opers}\label{s:trssl2}

We now explain a connection between nondegenerate twisted $(\SL(2),q)$-opers
and the two particle trigonometric
Ruijsenaars-Schneider model.  More precisely, we show that the
integrals of motion in the tRS model arise from nondegenerate twisted
opers with two regular singularities of weight one and with $Q_{-}$ linear.

Consider $Z$-twisted opers with two regular singularities $z_\pm$,
both of weight one, so $\rho=(z-z_+)(z-z_-)$.  For generic $q$, the
degree of the quantum Wronskian equals $\deg(Q_+)+\deg(Q_-)$.  Here,
we will only look at $q$-opers for which $\deg(Q_\pm)=1$, say
$Q_{-}=z-p_{-}$ and $Q_{+}=c(z-p_{+})$.  Here, $c$ is a nonzero
constant for which the quantum Wronskian is monic; an easy
calculation shows that $c=q^{-1}(\ze^{-1}-\ze)^{-1}$. 

Setting the quantum Wronskian equal to $\rho$ gives us the equation
\begin{equation}
z^2-\frac{z}{q}\left[\frac{\zeta-q\zeta^{-1}}{\zeta-\zeta^{-1}}p_+ + \frac{q\zeta-\zeta^{-1}}{\zeta-\zeta^{-1}}p_-\right]
+\frac{p_+ p_-}{q}
   =(z-z_+)(z-z_-)\,.
\end{equation}  Comparing powers of
$z$ on both sides, we obtain
\begin{equation}
\begin{gathered}\label{eq:tRS2body}\frac{\zeta-q\zeta^{-1}}{\zeta-\zeta^{-1}}p_+
    + \frac{q\zeta-\zeta^{-1}}{\zeta-\zeta^{-1}}p_- = q(z_+ + z_-)\\
\frac{p_+ p_-}{q} = z_+ z_-\,.
\end{gathered}
\end{equation}

Upon introducing coordinates $\zeta_{+}, \zeta_-$ such that
$\zeta=\zeta_+/\zeta_-$ and viewing $\zeta_{\pm}, p_{\pm}$ as the
positions and momenta in the two particle tRS model, we see that
\eqref{eq:tRS2body} are just the trigonometric Ruijsenaars-Schneider
equations~\cite{Koroteev:2017aa}.  In fact, the set of $Z$-twisted
opers with weight one singularities at $z_{\pm}$ is just the
intersection of two Lagrangian subspaces of the two particle tRS phase
space: the subspace determined by \eqref{eq:tRS2body} and the subspace
with the $\zeta_{\pm}$ fixed constants satisfying
$\zeta=\zeta_+/\zeta_-$.  As we will see in Section~\ref{Sec:Kth},
this construction can be generalized to higher rank.

\section{$(\SL(N),q)$-opers}\label{Sec:SLNqOps}
\subsection{Definitions}\label{sec:slnqdef}

We now discuss the generalization of $(\SL(2),q)$-opers to $\SL(N)$.

\begin{Def} A \emph{$(\GL(N),q)$-oper} on $\mathbb{P}^1$ is a triple
  $(E,A,\cL_\bullet)$, where $(E,A)$ is a $(\GL(N),q)$-connection and
  $\cL_\bullet$ is a complete flag of subbundles such that $A$ maps $\cL_i$
  into $\cL_{i+1}^q$ and the induced maps
  $\bar{A}_i:\cL_i/\cL_{i-1}\to \cL^q_{i+1}/\cL^q_i$ are
  isomorphisms for $i=1,\dots,N-1$. The triple is called an
  \emph{$\SL(N)$-oper} if  $(E,A)$ is an $(\SL(N),q)$-connection.
\end{Def}
To make this definition more explicit, consider the determinants
\begin{equation}\label{altqW} \Big(s(q^{i-1}z)\wedge
  A(q^{i-2}z)s(q^{i-2}z)\wedge\dots\wedge
  \Big(\prod_{j=0}^{i-2}(A(q^{i-2-j}z)\Big)s(z)\Big)\bigg|_{\Lambda^i
\cL_i^{q^{i-1}}}
\end{equation}
for $i=1,\dots, N$, where $s$ is a local section of $\cL_1$.  Then
$(E,A,\cL_\bullet)$ is a $q$-oper if and only if at every point, there
exists local sections for which each  $\cW_i(s)(z)$ is nonzero.  It
will be more convenient to consider determinants with the same zeros
as those in \eqref{altqW}, but with no $q$-shifts:
\begin{equation}\label{qW}\cW_i(s)(z)=\Big(s(z)\wedge
  A(z)^{-1}s(q z)\wedge\dots\wedge
  \Big(\prod_{j=0}^{i-2}(A(q^{j}z)^{-1}\Big)s(q^{i-1}z)\Big)\bigg|_{\Lambda^i
\cL_i}.
\end{equation}

As in the classical setting, we need to relax these conditions to
allow for regular singularities.  Fix a collection of $L$ points
$z_1,\dots,z_L\ne 0,\infty$ such that the $q^\Z$-lattices they
generate are pairwise disjoint.  We associate a dominant integral
weight $\lambda_m=\sum l_m^i\om_i$ to each $z_m$.  Set
$\ell_m^{i}=\sum_{j=1}^{i} l_m^{j}$.

\begin{Def} An \emph{$(\SL(N),q)$-oper with regular singularities} at
  the points $z_1, \dots, z_L\ne 0,\infty$ with weights $\lambda_1,
  \dots \lambda_L$ is a meromorphic $(\SL(N),q)$-oper such that each
  $\bar{A}_i$ is an isomorphism except at the points
  $q^{-\ell^{i-1}_m}z_m, q^{-\ell^{i-1}_m+1}z_m, \dots,
  q^{-\ell^{i}_m+1}z_m$ for each $m$, where it has simple zeros.
\end{Def}

\begin{figure}[!h]
\includegraphics[scale=0.35]{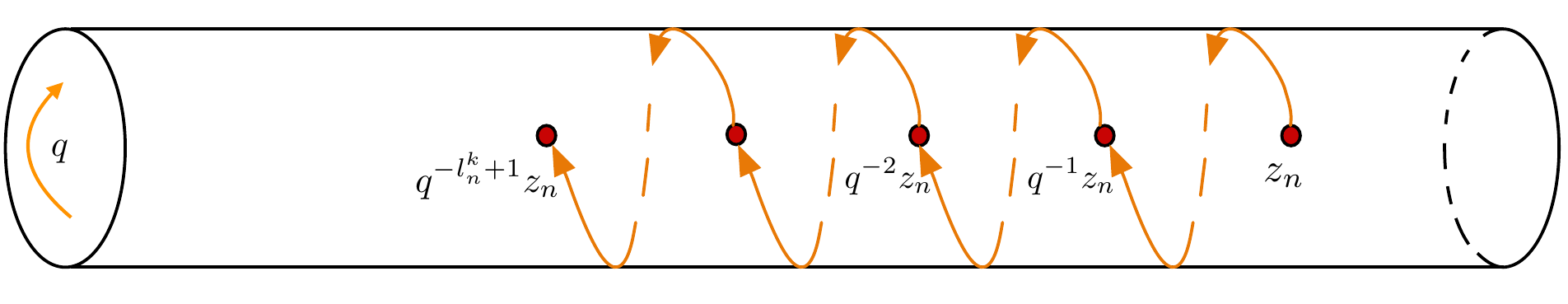}
\caption{Weight of the singularity $z_n$ as $q$-monodromy around the
  cylinder ($\mathbb{P}^1$ with $0$ and $\infty$ removed).}
\label{fig:qmon}
\end{figure}
 
In order to express the locations of the roots of the $\cW_i(s)$'s,
it is convenient to introduce the polynomials
\begin{equation}\label{LambdaPoly}
\Lambda_i =
\prod^{L}_{m=1}\prod^{\ell_m^{i}-1}_{j=\ell_m^{i-1}}(z-q^{-j}z_m)\,
\end{equation}
with zeros precisely where $\bar{A}_i$ is not an isomorphism.  We also
set \begin{equation}\label{Pidef} P_i=\Lambda_1 \Lambda_2\cdots\Lambda_i = \prod^{L}_{m=1}\prod^{\ell_m^{i}-1}_{j=0}(z-q^{-j}z_m).
\end{equation} 
We introduce the notation $f^{(j)}(z)=D_q^j(f)(z)=f(q^j z)$. The zeros
of $\cW_k(s)$ coincide with those of the polynomial
\begin{equation} \label{eq:Wkdef}
\begin{aligned}
W_k(s)&=\Lambda_1\left(\Lambda^{(1)}_1 \Lambda^{(1)}_2 \right)\cdots
 \left(\Lambda^{(k-2)}_1\cdots\Lambda^{(k-2)}_{k-1}\right)\\&= P_1 \cdot P^{(1)}_2\cdot P^{(2)}_3\cdots P^{(k-2)}_{k-1}\,.
\end{aligned}
\end{equation}

We now define \emph{twisted $q$-opers}.  Let
$Z=\diag(\zeta_1,\dots,\zeta_N)\in\SL(N,\C)$ be a diagonal matrix with
distinct eigenvalues.
 
\begin{Def} An $(\SL(N),q)$-oper $(E,A,\cL_\bullet)$ with regular singularities
  is called a \emph{$Z$-twisted $q$-oper} if $A$ is gauge-equivalent
  to $Z^{-1}$.
\end{Def}

 As in the $\SL(2)$ case, this is a deformed version of opers with
 irregular singularities that arise in the inhomogeneous
version of the Gaudin model introduced in
\cite{FFTL:2010,Rybnikov:2010}.

\subsection{Miura $q$-opers and quantum Wronskians}

Given a $q$-oper with regular singularities $(E,A,\cL_\bullet)$, we
can define the associated \emph{Miura $q$-opers} as quadruples
$(E,A,\cL_\bullet,\hcL_\bullet)$ where $\hcL_\bullet$ is a complete
flag preserved by the $q$-connection, i.e., $A$ maps $\hcL_i$ into
$\hcL^q_i$ for all $i$.  Again, we will primarily be interested in
\emph{nondegenerate} Miura $q$-opers.  This means that the flags
$(\cL_\bullet(z),\hcL_\bullet(z))$ are in general position at all but
a finite number of points $\{w_j\}$; moreover, at each $w_j$, the
relative position is $w_0 \sigma_k$ for some simple reflection
$\sigma_k$.  Finally, we assume that $q^\Z w_i\cap q^\Z
w_j=\varnothing$ if $i\ne j$ and also that the $q^\Z$ lattices
generated by the $z_m$'s and $w_j$'s do not intersect.  (We remark
that these last conditions are stronger than necessary; for example,
one may instead specify that $w_j\ne q^i z_m$ for all $j$ and $m$ and
for $|i|\le n$, where $n$ is a positive integer that may be computed
explicitly from the weights.)

We now specialize to the case where $(E,A,\cL_\bullet)$ is a
$Z$-twisted $q$-oper.  Here, there are only a finite number of possible
associated Miura $q$-opers.  Indeed, if we consider the gauge where
the matrix of the $q$-connection is the regular semisimple diagonal
matrix $Z^{-1}$, we see that the only possibilities for $\hcL_\bullet$
are the $N!$ flags given by the permutations of the standard ordered
basis $e_1,\dots,e_N$.  (This is analogous to the classical situation.
The Miura opers lying above a given oper with regular singularities
and trivial monodromy are parametrized by the flag manifold.  However,
there are only $N!$ Miura opers associated to an oper with regular
singularities on $\P^1\setminus\infty$ whose underlying connection if
$d+h\, dz$, where $h\in\gl(N,\C)$ is regular semisimple.) It
suffices to consider Miura $q$-opers for the standard flag; indeed, if
not, we can gauge change to one where $\hcL_\bullet$ is the standard
flag, but where $Z$ is replaced by a Weyl group conjugate.

Let $s(z)=(s_1(z),\dots,s_N(z))$ be a section generating $\cL_1$,
where the $s_a$'s are polynomials.  We
now show that the nondegeneracy of the Miura $q$-oper may be expressed
in terms of quantum Wronskians.  Consider the zeros of the determinants \begin{equation}\label{qD}
\mathcal{D}_k(s)=e_1\wedge\dots\wedge{e_{N-k}}\wedge
s(z)\wedge Z s(qz)\wedge\dots\wedge Z^{k-1}s(q^{k-1}z)\,
\end{equation}
for $k=1,\dots,N$. The arguments of Section~\ref{s:slnclassical} show
that for our $q$-oper to be nondegenerate, we need the zeros of
$\mathcal{D}_k(s)$ in $\cup_m q^\Z z_m$ to coincide with those of
$\cW_k(s)$.  Moreover, we want the other roots of $\mathcal{D}_k(s)$
to generate disjoint $q^\Z$ lattices.  To be more explicit, for
$k=1,\dots,N$, we have nonzero constants $\alpha_k$ and polynomials
\begin{equation} \mathcal{V}_k(z) = \prod_{a=1}^{r_k}(z-v_{k,a})\,,
\label{eq:BaxterRho}
\end{equation}
for which 
\begin{equation}
\det\begin{pmatrix} \,     1 & \dots & 0 & s_{1}(z) & \zeta_{1} s_{1}(qz) & \cdots & \zeta_{1}^{k-1} s_{1}(q^{k-1}z) \\ 
 \vdots & \ddots & \vdots& \vdots & \vdots & \ddots & \vdots \\  
0 & \dots & 1&s_{N-k}(z) & \zeta_{N-k} s_{N-k}(qz) &  \dots & \zeta^{k-1}_{N-k} s_{N-k}(q^{k-1}z)  \\  
0 & \dots & 0&s_{N-k+1}(z) & \zeta_{N-k+1} s_{N-k+1}(qz) &  \dots & \zeta^{N-k-1}_{N-k+1} s_{N-k+1}(q^{k-1}z)  \\
\vdots & \ddots & \vdots&\vdots & \vdots & \ddots & \vdots \\
0 & \dots & 0&s_{N}(z) & \zeta_{N} s_{N}(qz) & \cdots & \zeta_{N}^{k-1} s_{N}(q^{k-1}z)  \, \end{pmatrix} =\alpha_{k} W_{k}
\cV_{k} \,; 
\label{eq:MiuraQOperCond}
\end{equation}
moreover $q^\Z v_{k,a}$ is disjoint from every other $q^\Z v_{i,b}$
and each $q^\Z z_m$.  Since $\cD_N(s)=\cW_N(s)$, we have
$\cV_N=1$.  We also set $\cV_0=1$; this is consistent with the fact
that \eqref{qD} also makes sense for $k=0$, giving
$\cD_0=e_1\wedge\dots\wedge e_N$.

We can also rewrite \eqref{eq:MiuraQOperCond} as
\begin{equation}
  \underset{i,j}{\det} \left[\zeta_{N-k+i}^{j-1} s_{N-k+i}^{(j-1)}\right] = \alpha_{k} W_{k} \mathcal{V}_{k}\,,
\label{eq:MiuraDetForm}
\end{equation}
where $i,j = 1,\dots,k$.

We remark that the nonzero constants $\a_1,\dots,\a_N$ are
normalization constants for the section $s$ and may be chosen
arbitrarily by first multiplying $s$ by a nonzero constant and then
applying constant gauge changes by diagonal matrices in $\SL(N)$.

\subsection{$(\SL(N),q)$-Opers and the XXZ Bethe ansatz}
We are now ready to state and prove our main theorem which relates twisted
$(\SL(N),q)$-opers to solutions of the XXZ Bethe ansatz equations for $\sl_N$. 

The Bethe equations for the general $\sl_N$ XXZ spin chain depend on
an anisotropy parameter $q\in\C^*$ and twist parameters
$\kappa_1,\dots,\kappa_N$ satisfying $\prod \kappa_i=1$.  The
equations can be written in the following form
\begin{equation}\label{eq:BetheExplicit}
 \frac{\kappa_{k+1}}{\kappa_{k}}\prod_{s=1}^L\frac{q^{\ell_s^{k}+\frac{k}{2}-\frac{3}{2}}\,u_{k,a}-z_s}{q^{\ell_s^{k-1}+\frac{k}{2}-\frac{3}{2}}\,u_{k,a}-z_s}\cdot\prod_{c=1}^{
 r_{k-1}}\frac{q^{\frac{1}{2}}u_{k,a}-u_{k-1,c}}{q^{-\frac{1}{2}}
 u_{k,a}-u_{k-1,c}}\cdot\prod_{b=1}^{r_k}\frac{q^{-1}
 u_{k,a}-u_{k,b}}{qu_{k,a}-
 u_{k,b}}\cdot\prod_{d=1}^{r_{k+1}}\frac{q^{\frac{1}{2}}u_{k,a}-u_{k+1,d}}{q^{-\frac{1}{2}} u_{k,a}-u_{k+1,d}}=1\,
\end{equation}
for $k=1,\dots N-1$, $a=1,\dots,r_k$.  (See, for example, \cite{Reshetikhin:1987}.)
The constants $\ell^i_m$ are determined by the dominant weights
$\lam_1,\dots,\lam_{L}$ as in Section~\ref{sec:slnqdef}.  We use the
convention that $r_0=r_N=0$, so one of the products in the first and
last equations is empty.

We remark that there exist many different
  normalizations of the XXZ Bethe equations in the literature depending
  on the scaling of the twist parameters.  The present normalization is
  designed to match the formulas obtained from $q$-opers.

We say that a solution of the Bethe equations is nondegenerate if
$z_s\notin q^{\frac{1-k}{2}}q^{\Z}u_{k,a}$ for all $k$ and $a$ and also
that $u_{k,a}\notin q^{\frac{k-k'}{2}}q^{\Z}u_{k',a'}$ unless $k=k'$
and $a=a'$.

\begin{theorem}\label{eq:MainTheorem} Suppose that
  $\kappa_1,\dots,\kappa_N$ generate disjoint $q^\Z$-lattices.  Then,
  there is a one-to-one correspondence between nondegenerate solutions
  of the $\sl_N$ XXZ Bethe ansatz equations \eqref{eq:BetheExplicit} with twist parameters
  $\kappa_i$ and nondegenerate $Z$-twisted $(\SL(N),q)$-opers with regular
  singularities at $z_1,\dots,z_L$ with dominant weight
  $\lam_1,\dots,\lam_L$ provided that
\begin{equation}
q^{\frac{1-k}{2}}u_{k,a} = v_{k,a}\quad\text{ and }\qquad \zeta_{k} = \kappa_{N+1-k}\,
\label{eq:rescalings}
\end{equation}
for $k=1,\dots, N$.
Moreover, the $q$-oper equations \eqref{eq:MiuraQOperCond} become
identical to the Bethe equations if one normalizes the section $s$ via
 \begin{equation}\alpha_k = q^{\frac{k-1}{2}r_k}\det V(\kappa_k,\dots,\kappa_1).
 \end{equation}
\end{theorem}

For the computations to follow, it will be convenient to introduce the
Baxter polynomials\footnote{The terminology
  comes from the analogy between these polynomials and the polynomials
  determining the eigenvalues of the Baxter operators arising from
  transfer matrices~\cite{Reshetikhin:2010si}.}
\begin{equation}\label{eq:PiDef}
\Pi_k =p_k\prod_{s=1}^L \prod_{j=\ell_s^{k-1}}^{\ell_s^{k}-1}\left(z-q^{1-\frac{k}{2}-j} z_s\right)\,,\qquad
Q_k = \prod_{a=1}^{r_k}(z-u_{k,a})\,, \qquad k= 1,\dots N-1\,,
\end{equation}
where the normalization constants $p_k =q^{(\frac{k}{2}-1)\sum_{m=1}^L
  l_{k}}$ are chosen so that
$\Pi_k=\Lambda^{\tiny{(\frac{k}{2}}-1)}_{k}$.  The Bethe equations
\eqref{eq:BetheExplicit} can then be written as
\begin{equation}\label{eq:TQBaxter}
\left.\frac{\kappa_{k+1}}{\kappa_{k} } \frac{\Pi_k^{(\tiny{\frac{1}{2}})}Q_{k-1}^{(\tiny{\frac{1}{2}})} Q_{k}^{(-1)} Q_{k+1}^{(\tiny{\frac{1}{2}})}}{\Pi_k^{(-\tiny{\frac{1}{2}})}Q_{k-1}^{(-\tiny{\frac{1}{2}})} Q_{k}^{(1)} Q_{k+1}^{(-\tiny{\frac{1}{2}})}}\right|_{u_{k,a}} = -1\,,
\end{equation}
where we recall that $f^{(p)}(z)=f(q^p z)$.

We observe that the Baxter polynomials are remarkably similar to the
polynomials $\Lambda_k$ and $\cV_k$ (see \eqref{LambdaPoly} and
\eqref{eq:BaxterRho}) which we used to describe the zeros of the
quantum Wronskians arising from twisted $q$-opers.  Our main theorem
makes this connection precise.

In order to prove the theorem, we will need four lemmas.

\begin{lemma}\label{eq:PropQQtilde}
Suppose that $\kappa_{k}\notin q^{\mathbb{N}_0}\kappa_{k+1}$ for all $k$.
Then, the system of equations \eqref{eq:TQBaxter}
is equivalent to the existence of auxiliary polynomials  $\widetilde
Q_k(z)$ satisfying the following system of equations
\begin{equation}\label{eq:QQrelations}
\kappa_{k+1} Q^{(-\tiny{\frac{1}{2}})}_k \widetilde Q^{(\tiny{\frac{1}{2}})}_k - \kappa_{k} Q^{(\tiny{\frac{1}{2}})}_k \widetilde Q^{(-\tiny{\frac{1}{2}})}_k= (\kappa_{k+1}-\kappa_k)Q_{k-1}Q_{k+1}\Pi_k\,,
\end{equation}
for $ k= 1,\dots N-1$.  Moreover, these polynomials are unique.
\end{lemma}

\begin{proof}  Set $g(z)=\widetilde{Q}_k(z)/Q_k(z)$ and
  $f(z)=(\kappa_{k+1}-\kappa_k)Q_{k-1}^{(\frac{1}{2})}Q_{k+1}^{(\frac{1}{2})}\Pi_k^{(\frac{1}{2})}$, so that \eqref{eq:QQrelations} may be rewritten as 
\begin{equation}\label{e:gf}
\kappa_{k+1}g^{(1)}_k(z)-\kappa_kg_k(z)=\frac{f(z)}{Q_{k}(z)Q^{(1)}_k(z)}.
\end{equation}
We then have the partial fraction decompositions
\begin{equation} \begin{aligned}
&\frac{f(z)}{Q_{k}(z)Q^{(1)}_k(z)}=h(z)-\sum_a\frac{b_a}{z-u_{k,a}}+\sum_a\frac{c_a}{qz-u_{k,a}},\\
&g_k(z)=\tilde{g}_k(z)+\sum_a\frac{d_a}{z-u_{k,a}}
\end{aligned}
\end{equation}
where $h(z)$ and $\tilde{g}_k(z)$ are polynomials.  In order for the
residues at each $u_{k,a}$ to match on the two sides of \eqref{e:gf},
one needs
\begin{equation}d_a=\frac{b_a}{\kappa_{k}}=\frac{c_a}{\kappa_{k+1}}.
\end{equation}
The second equality is merely the Bethe equations \eqref{eq:TQBaxter} in
the alternate form
\begin{equation}
 \Res_{u_{k,a}}\left[\frac{f(z)}{\kappa_{k} Q_k(z)Q^{(1)}_k(z)}\right]+ \Res_{u_{k,a}}\left[\frac{f^{(-1)}(z)}{\kappa_{k+1} Q^{(-1)}_k(z)Q_k(z)}\right]=0
\end{equation}
or
\begin{equation}
\left(\left.\frac{Q^{(\tiny{\frac{1}{2}})}_{k-1}Q^{(\tiny{\frac{1}{2}})}_{k+1}\Pi_k^{(\tiny{\frac{1}{2}})}}{\kappa_{k}
    Q^{(1)}_k}+\frac{Q^{(-\tiny{\frac{1}{2}})}_{k-1}Q^{(-\tiny{\frac{1}{2}})}_{k+1}\Pi_k^{(-\tiny{\frac{1}{2}})}}{\kappa_{k+1}Q^{(-1)}_k}\right)\right|_{u_{k,a}}=0.
\end{equation} 
Next, to solve for the polynomial $\tilde{g}_k(z)$, set
$\tilde{g}_k(z)=\sum r_iz^i$ and $h(z)=\sum s_i z^i$.  We then obtain
the equations $r_i(\kappa_{k+1}q^i -\kappa_k )=s_i$.  Our assumptions
on the $\kappa_j$'s imply that these equations are always solvable.
Thus, there exist polynomials $\widetilde{Q}_k(z)$ satisfying
\eqref{eq:QQrelations} if and only if the Bethe equations hold.  The
uniqueness statement holds since the solutions for the residues $d_a$
and the coefficients of the polynomial $\tilde{g}_k(z)$ are unique.

\end{proof}

\begin{lemma}\label{eq:PropQQtilde}
The system of equations \eqref{eq:QQrelations} is equivalent to the set of equations
\begin{equation}
\kappa_{k+1} \mathscr{D}^{(-\tiny{\half})}_k \widetilde{\mathscr{D}}^{(\tiny{\half})}_k - \kappa_{k} \mathscr{D}^{(\tiny{\half})}_k \widetilde{\mathscr{D}}^{(-\tiny{\half})}_k= (\kappa_{k+1}-\kappa_k)\mathscr{D}_{k-1}\mathscr{D}_{k+1}\,,
\label{eq:QQVform}
\end{equation}
for the polynomials 
\begin{equation}\label{eq:Dkdef}
\mathscr{D}_k={Q_k}{F_k}\,,\qquad \widetilde{\mathscr{D}}_k={\widetilde Q_k}{F_k}\,,
\end{equation}
where $F_k=W^{(\tiny{\frac{1-k}{2}})}_{k}$.
\end{lemma}

\begin{proof} The $F_k$'s are solutions to the functional equation
\begin{equation}
\frac{F_{k-1} \cdot F_{k+1}}{F^{(\tiny{\half})}_k \cdot F_k^{(-\tiny{\half})}}=\Pi_k\,.
\label{eq:FFuncEq}
\end{equation}
Indeed, since $W_k=P^{\tiny{(k-2)}}_{k-1}W_{k-1}$ and $P_k=\Lambda_k P_{k-1}$,
we have
\begin{equation}
\frac{F_{k-1} \cdot F_{k+1}}{F^{(\tiny{\half})}_k \cdot F_k^{(-\tiny{\half})}}=
\frac{W^{\tiny{(\frac{2-k}{2})}}_{k-1} \cdot W^{\tiny{(-\frac{k}{2}})}_{k+1}}
{W^{(\tiny{\frac{2-k}{2}})}_k \cdot
  W_k^{(-\tiny{\frac{k}{2}})}}=\frac{P^{\tiny{(\frac{k}{2}-1)}}_{k}}{P^{\tiny{(\frac{k}{2}-1)}}_{k-1}}=\Lambda^{\tiny{(\frac{k}{2}-1})}_{k}=\Pi_k.
\end{equation}
The equivalence of \eqref{eq:QQrelations} and \eqref{eq:QQVform}
follows easily from this fact.

\end{proof}

Let $V(\gamma_1,\dots,\gamma_k)$ denote the $k\times k$ Vandermonde
matrix $(\gamma_i^j)$.  We recall that this determinant is nonzero if
and only if the $\gamma_i$'s are distinct.

\begin{lemma}\label{Th:existencePoly}
Suppose that $\gamma_1,\dots,\gamma_{k-1}$ are nonzero complex numbers
such that $\gamma_j\notin
  q^{\mathbb{N}_0}\gamma_{k}$ for $j<k$.  Let $f_1,\dots,f_{k-1}$ be
  polynomials that do not vanish at $0$,
  and let $g$ be an arbitrary polynomial.
   Then there exists a unique polynomial $f_k$ satisfying 
\begin{equation}\label{Fmatrix}
g=\det\begin{pmatrix} f_1 & \gamma_{1} f_1^{(1)} & \cdots &
  \gamma_{1}^{k-1} f_1^{(k-1)} \\ \vdots & \vdots & \ddots & \vdots \\
  f_k  & \gamma_{k} f_k^{(1)} & \cdots & \gamma_k^{k-1} f_k^{(k-1)} \,
\end{pmatrix}.  
\end{equation}
Moreover, if $g(0)\ne 0$, then $f_k(0)\ne 0$.
 \end{lemma}
 \begin{proof} Set $f_j(z)=\sum a_{ji} z^i$ and $g(z)=\sum b_i z^i$,
   and let $F$ denote the matrix in \eqref{Fmatrix}.  We show that we
   can solve for the $a_{ki}$'s recursively.  Expanding by minors
   along the bottom row, we get $g=\sum_{j=1}^k (-1)^{k+j}\det
   F_{k,j}f_k^{(j-1)}$.  First, we equate the constant terms.  This
   gives
   $$b_0=a_{k0}\left(\prod_{j=1}^{k-1}a_{j0}\right)\sum_{j=1}^k  (-1)^{k+j}\gamma_k^{j-1}\det
   V(\gamma_1,\dots,\gamma_k)_{k,j}=a_{k0}\left(\prod_{j=1}^{k-1}a_{j0}\right)\det
   V(\gamma_1,\dots,\gamma_k).$$ Since the $\gamma_j$'s are distinct,
   the Vandermonde determinant is nonzero.  Moreover, $a_{j0}\ne 0$
   for $j=1,\dots, k-1$.  Thus, we can solve uniquely for $a_{k0}$.
   In particular, if $b_0=0$, then $a_{k0}=0$.

   For the inductive step, assume that we have found unique $a_{kr}$
   for $r<s$ such that the polynomial equation \eqref{Fmatrix} has
   equal coefficients up through degree $s-1$.  We now look at the
   coefficient of $z^s$.  The only way that $a_{ks}$ appears in this
   coefficient is through the constant terms of the minors $F_{k,j}$.
   To be more explicit, equating the coefficient of $z^s$ in
   \eqref{Fmatrix} expresses $c a_{ks}$ as a polynomial in known
   quantities, where
\begin{align*} c&=\left(\prod_{j=1}^{k-1}a_{j0}\right)\sum_{j=1}^k  (-1)^{k+j}(q^{s}\gamma_k)^{j-1}\det
   V(\gamma_1,\dots,\gamma_{k-1},q^s\gamma_k)_{k,j}\\&=\left(\prod_{j=1}^{k-1}a_{j0}\right)\det
   V(\gamma_1,\dots,\gamma_{k-1},q^s\gamma_k).
 \end{align*}
Again, our condition on the $\gamma_j$'s implies that the Vandermonde
determinant is nonzero, so there is a unique solution for $a_{ks}$. 
 \end{proof}
 
In the following lemma, we consider matrices 
\begin{equation}
  M_{i_1,\dots, i_j}=\begin{pmatrix} \,  \mathfrak{q}_{i_1}^{(\small{\frac{1-j}{2})}} & \kappa_{N+1-i_1} \mathfrak{q}_{i_1}^{(\small{\frac{3-j}{2})}} & \cdots & \kappa_{N+1-i_1}^{j-1} \mathfrak{q}_{i_1}^{{(\small{\frac{j-1}{2})}}} \\ \vdots & \vdots & \ddots & \vdots \\  \mathfrak{q}_{i_j}^{(\small{\frac{1-j}{2})}}  & \kappa_{N+1-i_j} \mathfrak{q}_{i_j}^{(\small{\frac{3-j}{2})}} & \cdots & \kappa_{N+1-i_j}^{j-1} \mathfrak{q}_{i_j}^{(\small{\frac{j-1}{2})}} \, \end{pmatrix}\\,
\label{eq:MiuraQOperCond1q}
\end{equation}
where $\mfq_1,\dots,\mfq_N$ are polynomials.  We also set $V_{i_1,\dots,i_j}=V(\kappa_{N+1-i_1},\dots,\kappa_{N+1-i_j})$.

\begin{lemma}\label{eq:qWronskiansShift}
Assume that the lattices $q^\Z \kappa_k$ are disjoint for distinct
$k$.  Given polynomials $\mathscr{D}_k, \widetilde{\mathscr{D}}_k$ for $k=1,\dots,N-1$ satisfying
\eqref{eq:QQVform}, there exist unique polynomials
$\mfq_1,\dots,\mfq_N$ such that 
\begin{equation}
\mathscr{D}_k = \frac{\det\, M_{N-k+1,\dots, N}}{\det\,
  V_{N-k+1,\dots, N}}\,\qquad\text{ and } \qquad
\widetilde{\mathscr{D}}_k = \frac{\det\, M_{N-k,N-k+2,\dots, N}}{\det\, V_{N-k,N-k+2,\dots, N}}\,.
\label{eq:MM0Inda}
\end{equation}
\end{lemma}

For future reference, we note that the first relations from \eqref{eq:MM0Inda} can be rewritten
as
\begin{equation}
 \det \left[\kappa_{k+1-i}^{j-1} \mathfrak{q}_{N-k+i}^{\left(j-\frac{k+1}{2}\right)}\right]_{1\le i,j\le k} = \det\left[\kappa_{k+1-i}^{j-1}\right]_{1\le i,j\le k} \mathscr{D}_k\,.
\label{eq:BetheDetForm}
\end{equation}

\begin{proof}

  We begin by observing that since $W_k$ and $Q_k$ do not vanish at
  $0$, $\sD_k(0)\ne 0$ for all $k$.  This implies that
  $\widetilde{\sD}_k(0)\ne 0$ for all $k$ as well; otherwise, by
  \eqref{eq:QQVform}, either $\sD_{k-1}$ or $\sD_{k+1}$ would vanish
  at $0$.

Now, set $\mfq_N=\sD_1$ and $\mfq_{N-1}=\widetilde{\sD}_1$.  It is obvious
that these are the unique polynomials satisfying \eqref{eq:MM0Inda}
for $k=1$ and that $\mfq_N(0), \mfq_{N-1}(0)\ne 0$.  Also,
\eqref{eq:QQVform} gives
\begin{equation}
\kappa_{2} \mathfrak{q}_{N}^{(-\tiny{\frac{1}{2}})} \mathfrak{q}_{N-1}^{(\tiny{\frac{1}{2}})} - \kappa_{1} \mathfrak{q}_N^{(\tiny{\frac{1}{2}})} \mathfrak{q}_{N-1}^{(-\tiny{\frac{1}{2}})} = (\kappa_{2}-\kappa_{1})\,\mathscr{D}_2\,
\end{equation}
so $\sD_2=M_{N-1,N}/V_{N-1,N}$.

Next, suppose that for $2\le k\le N-1$, we have shown that there exist
unique polynomials $\mfq_N,\dots,\mfq_{N-k+1}$ such the formulas for
$\sD_j$ (resp. $\widetilde{\sD}_j$) in \eqref{eq:MM0Inda} hold for
$1\le j\le k$ (resp. $1\le j\le k-1$).  Furthermore, assume that none
of these polynomials vanish at $0$.  We will show that there exists a
unique $\mfq_{N-k}$ such that the formulas for $\sD_{k+1}$ and
$\widetilde{\sD}_{k}$ hold and that $\mfq_{N-k}(0)\ne 0$.  This will
prove the lemma.

We use Lemma~\ref{Th:existencePoly} to define $\mfq_{N-k}$.  In the
notation of that lemma, set $f_j=\mfq_{N+1-j}^{(\frac{1-k}{2})}$ and
$\gamma_j=\kappa_j$ for $1\le j\le k-1$, and set
$g=(-1)^{\frac{k(k-1)}{2}}(\det
V_{N-k,N-k+2,\dots,N})\widetilde{\sD}_{k}$.  (The sign factor occurs
because we have written the rows in reverse order to apply the lemma.)
By hypothesis, $f_j(0)\ne 0$ for $1\le j\le k-1$, so there exists a
unique $f_{k}$ satisfying \eqref{Fmatrix}.  Moreover, $g(0)\ne 0$,
so $f_{k}\ne 0$.  It is now clear that
$\mfq_{N-k}=f_{k-1}^{(\frac{k-1}{2})}$ is the unique polynomial
satisfying the formula in \eqref{eq:MM0Inda} for
$\widetilde{\sD}_{k}$.  Of course, $\mfq_{N-k}\ne 0$.

To complete the inductive step, it remains to show that the formula
for $\sD_{k+1}$ is satisfied.  We make use of the Desnanot-Jacobi/Lewis
Carroll identity for determinants.  Given a square matrix $M$, let $M^i_j$ denote
the square submatrix with row $i$ and column $j$ removed;
similarly, let $M^{i,i'}_{j,j'}$ be the submatrix with rows $i$ and
$i'$ and columns $j$ and $j'$ removed.  We will apply this identity in
the form
\begin{equation}\label{eq:JacobiNew}
  \det M^{1}_1 \det M^{2}_{k+1}- \det M^{1}_{k+1} \det M^{2}_{1}= \det
  M^{1,2}_{1,k+1} \det M\,.
\end{equation}

Set $M=M_{N-k,\dots,N}$.  All of the matrices appearing in
\eqref{eq:JacobiNew} are obtained from matrices of the form
\eqref{eq:MiuraQOperCond1q} via $q$-shifts, multiplication of each row
by an appropriate $\kappa_i$, or both.  In particular,
$M^1_{k+1}=M_{N-k+1,\dots,N}^{(-\frac{1}{2})}$ and
$M^{2}_{k+1}=M_{N-k,N-k+2,\dots,N}^{(-\frac{1}{2})}$ while the
determinants of the other three are given by
\begin{equation}\begin{gathered} \det M^{1}_{1}=\kappa_{k}\bigg(\prod_{j=1}^{k-1}\kappa_j\bigg)
  \det M_{N-k+1,\dots,N}^{(\frac{1}{2})},\qquad \det M^{2}_{1}=\kappa_{k+1}\bigg(\prod_{j=1}^{k-1}\kappa_j\bigg)
  \det M_{N-k,N-k+2,\dots,N}^{(\frac{1}{2})},\\
 \det M^{1,2}_{1,k+1}=\bigg(\prod_{j=1}^{k-1}\kappa_j\bigg)
  \det M_{N-k+2,\dots,N}.
\end{gathered}
\end{equation}
Upon substituting into \eqref{eq:JacobiNew} and dividing by
$\prod_{j=1}^{k-1}\kappa_j$, we obtain

\begin{multline}\kappa_k \det M_{N-k+1,\dots,N}^{(\frac{1}{2})}\det
  M_{N-k,N-k+2,\dots,N}^{(-\frac{1}{2})}- \kappa_{k+1}\det M_{N-k+1,\dots,N}^{(-\frac{1}{2})}\det M_{N-k,N-k+2,\dots,N}^{(\frac{1}{2})}\\
 = \det M_{N-k+2,\dots,N} \det M_{N-k,\dots,N}\,.
\end{multline}
Finally, dividing both sides by $V_{N-k+1,\dots,N}V_{N-k,N-k+2,\dots,N}$ and
applying the inductive hypothesis gives \eqref{eq:Dkdef} multiplied by
$-1$.  This is obvious for the left-hand sides.  To see that the other sides match, one need only observe that
\begin{equation}\begin{aligned}V_{N-k+2,\dots,N}V_{N-k,\dots,N}&=(\kappa_k-\kappa_{k+1})\prod_{1\le i< j\le k-1}
  (\kappa_i-\kappa_j)^2\prod_{i=1}^{k-1}(\kappa_{i}-\kappa_{k})(\kappa_i-\kappa_{k+1})\\&=(\kappa_k-\kappa_{k+1})V_{N-k+1,\dots,N}V_{N-k,N-k+2,\dots,N}.
\end{aligned}
\end{equation}
\end{proof}


We are finally ready to prove Theorem~\ref{eq:MainTheorem}.


\begin{proof}[Proof of Theorem~\ref{eq:MainTheorem}]

We have shown that a solution to the Bethe equations is uniquely
determined by polynomials $\mfq_k$ satisfying
\eqref{eq:BetheDetForm}.  We will show that after matching the
parameters as in the statement and normalizing the section $s(z)$
generating the $q$-oper, the components $s_k$ also satisfy these
equations, so $s_k=\mfq_k$ for all $k$. Since the twisted $q$-oper is
uniquely determined by $s$, we obtain the desired correspondence.

After shifting \eqref{eq:BetheDetForm} by $\frac{k-1}{2}$ and
using the definition of $\sD_k$ from \eqref{eq:Dkdef}, we obtain
the equivalent form
\begin{equation}
 \underset{i,j}{\det} \left[\kappa_{k+1-i}^{j-1}
   \mathfrak{q}_{N-k+i}^{(j-1)}\right] = \underset{i,j}{\det}
 \left[\kappa_{k+1-i}^{j-1}\right] W_k Q_k^{(\frac{k-1}{2})}.
\end{equation}
On the other hand, rewriting the $q$-oper relations
\eqref{eq:MiuraDetForm} for convenience, we have
\begin{equation}
  \underset{i,j}{\det} \left[\zeta_{N-k+i}^{j-1} s_{N-k+i}^{(j-1)}\right] = \alpha_{k} W_{k} \mathcal{V}_{k}\,.
\end{equation}
If we set $q^{\frac{1-k}{2}}u_{k,a} = v_{k,a}$, then the roots of
$\cV$ and  $Q_k^{(\frac{k-1}{2})}$ coincide; moreover, the leading
terms of the polynomials on the right are the same if one takes
$\alpha_k = q^{\frac{k-1}{2}r_k}\det V(\kappa_k,\dots,\kappa_1)$.
Thus, if one sets $\zeta_{k} = \kappa_{N+1-k}$, the two equations are
identical.

It only remains to observe that the notions of nondegeneracy are
preserved by the transformation \eqref{eq:rescalings}.
\end{proof}

\section{Explicit equations for $(\SL(3),q)$-opers}\label{Sec:SL3qOps}
\subsection{A canonical form} 

In this section, we illustrate the general theory in the case of
$\SL(3)$.  In particular, we show that the underlying $q$-connection
can be expressed entirely in terms of the Baxter polynomials and the
twist parameters.

We start in the gauge where the connection is given by the diagonal
matrix $\diag(\zeta_1^{-1},\zeta_2^{-1},\zeta_3^{-1})$ and the section
generating the line bundle $\cL_1$ is $s=(s_1,s_2,s_3)$.  We now apply
a $q$-gauge change by a certain matrix $g(z)$ mapping $s$ to the
standard basis vector $e_3$:
\begin{equation}
g(z)=\begin{pmatrix}
 \beta(z) & -\alpha(z) & 0 \\
 0 & \beta(z)^{-1} & 0 \\
 0 & 0 & 1
\end{pmatrix}
\begin{pmatrix}
 s_2(z) & -s_1(z) & 0 \\
 0 & \frac{s_3(z)}{s_2(z)} & -1 \\
 0 & 0 & \frac{1}{s_3(z)}
\end{pmatrix}\,,
\end{equation}
where $\alpha(z)=\zeta^{-1}_1 s_1^{(-1)}s_2-\zeta^{-1}_2 s_1
s_2^{(-1)}$ and $\beta(z)=\frac{1}{s_2}(\zeta^{-1}_2 s_2^{(-1)} s_3
-\zeta^{-1}_3 s_2 s_3^{(-1)})$.
Applying the $q$-change formula \eqref{qgauge} leads to a matrix all of
whose entries are expressible in terms of minors of the matrix 
\begin{equation}
M_{1,2,3}^{(1)}=\begin{pmatrix} s_1 & \zeta_1 s_1^{(1)} & \zeta_1^2 s_1^{(2)}\\s_2 & \zeta_2 s_2^{(1)} & \zeta_2^2 s_2^{(2)}\\s_3 & \zeta_3 s_3^{(1)} & \zeta_3^2 s_3^{(2)}
\end{pmatrix}.
\end{equation}
By \eqref{eq:MiuraQOperCond}, the relations between the Baxter
polynomials and these determinants are given by  
\begin{equation}\begin{gathered}
\det M_3=\alpha_1\cV_1^{(-1)},\qquad M_{2,3}=\alpha_2 W_2^{(-1)} \cV_2^{(-1)}=\alpha_2 \Lambda_1^{(-1)}
\cV_2^{(-1)},\quad\text{ and}\\
\det M_{1,2,3}=\alpha_3
W_3^{(-1)}=\alpha_3\Lambda_1^{(-1)}\Lambda_1\Lambda_2.
\end{gathered}
\end{equation}
A further diagonal $q$-gauge change by
$\diag(\alpha_3^{-2/3}(\Lambda^{(-1)}_1)^{-1},\alpha_3^{1/3}\Lambda^{(-1)}_1,\alpha_3^{1/3})$
brings us to our desired form:
\begin{equation}
A(z)=\begin{pmatrix}
a_1(z) & \Lambda_2(z) & 0\\
0 & a_2(z) & \Lambda_1(z) \\
0 & 0 & a_3(z)
\end{pmatrix}\,,
\label{eq:SL3qconn}
\end{equation}
where
\begin{equation}\label{eq:a123}\begin{aligned}
a_1&=\zeta^{-1}_1\frac{\Lambda^{(-1)}_1}{\Lambda_1}\cdot\frac{\det
    M_{2,3}}{\det M_{2,3}^{(-1)}}=\zeta^{-1}_1\,\frac{\cV_2}{\cV_2^{(-1)}}
\,,\cr
a_2&=\zeta^{-1}_2\frac{\Lambda_1}{\Lambda^{(-1)}_1}\cdot
\frac{s_3^{(1)}}{s_3}\frac{\det M_{2,3}^{(-1)}}{\det M_{2,3}}
=\zeta^{-1}_2\,\frac{\cV_2^{(-1)}}{\cV_2}\cdot\frac{\cV_1^{(1)}}{\cV_1}\,,\cr
a_3&=\zeta^{-1}_3\frac{s_3}{s^{(1)}_3}
=\zeta^{-1}_3\,\frac{\cV_1}{\cV_1^{(1)}}
\,.
\end{aligned}
\end{equation}
Note that the singularities of the oper and the Bethe roots can be
determined from the zeros of
the superdiagonal and the diagonal respectively.

\subsection{Scalar difference equations and eigenvalues of transfer
  matrices}

The first-order system of difference equations $f(qz)=A(z)f(z)$
determined by \eqref{eq:SL3qconn} can be expressed as a third-order
scalar difference equation.  This is accomplished by the $q$-Miura
transformation: a $q$-gauge change by a lower triangular matrix which
reduces $A(z)$ to companion matrix form.  (This procedure appears as
part of the difference equation version of Drinfeld-Sokolov reduction
introduced in \cite{Frenkel_1998,Frenkel:ls,1998CMaPh.192..631S}.) 

In the XXZ model, the eigenvalues of the $\SL(3)$-transfer matrices
for the two fundamental weights are  \cite{Bazhanov_2002,Frenkel:2013uda}
\begin{equation}\begin{aligned}
    T_1 &=a^{(2)}_1\Lambda_1 \Lambda^{(1)}_2+a^{(1)}_2 \Lambda_1
    \Lambda^{(2)}_2+a_3 \Lambda_1^{(1)} \Lambda^{(2)}_2 \,,\cr T_2 &=
    a_1^{(1)} a^{(1)}_2 \Lambda_1 \Lambda_2 +a_1^{(1)} a_3
    \Lambda^{(1)}_1 \Lambda_2 +a_2 a_3 \Lambda_1^{(1)}
    \Lambda^{(1)}_2\,.
\end{aligned}
\end{equation}
Just as for $\SL(2)$, these eigenvalues appear in the  coefficients of the scalar difference
equation associated to our twisted $q$-oper.  Indeed, a simple
calculation shows that the system  

\begin{equation} \begin{aligned}
 f^{(1)}_1&=a_1 f_1 + \Lambda_2 f_2\\ 
 f^{(1)}_2&=a_2 f_2 + \Lambda_1 f_3\\ 
 f^{(1)}_3&=a_3 f_3
\end{aligned}
\end{equation}
is equivalent to
\begin{equation}
\Lambda_1 \Lambda_2 \Lambda_2^{(1)}\cdot f_1^{(3)}
-\Lambda_2\, T_1\cdot f_1^{(2)}
+\Lambda_2^{(2)}\, T_2\cdot f_1^{(1)}
-\Lambda_1^{(1)} \Lambda_2^{(1)} \Lambda_2^{(2)} \cdot f_1 = 0\,.
\end{equation}

\section{Scaling limits: From $q$-opers to opers}\label{Sec:Limits}

In this section, we consider classical limits of our results.  We will
take the limit from $q$-opers to opers in two steps.  The first
will give rise to a correspondence between the spectra of a twisted
version of the XXX spin chain and a twisted analogue of the discrete
opers of \cite{Mukhin_2005}.  By taking a further limit, we recover
the relationship between opers with an irregular singularity and the
inhomogeneous Gaudin model \cite{FFTL:2010,Rybnikov:2010}.

First, we introduce an exponential reparameterization of $q$, the
singularities, and the Bethe roots: $q=e^{R\upepsilon}$, $z_s = e^{R
  \upsigma_s}$, and $v_{k,a}= e^{R \upupsilon_{k,a}}$.  We also set
$\tilde \ell_s^k = \ell_s^k +\frac{k}{2}-\frac{3}{2}$.  We now take
the limit of the XXZ Bethe equations \eqref{eq:BetheExplicit} as $R$
goes to $0$.  This limit brings us to the XXX Bethe equations
\begin{equation}
  \frac{\kappa_{k+1}}{\kappa_{k}}\prod_{s=1}^L\frac{\upupsilon_{k,a}+\ell_s^{k}\upepsilon-\upsigma_s}{\upupsilon_{k,a}+\ell_s^{k-1}\upepsilon-\upsigma_s}\cdot\prod_{c=1}^{r_{k-1}}\frac{\upupsilon_{k,a}-\upupsilon_{k-1,c}+\frac{1}{2}\upepsilon}{\upupsilon_{k,a}-\upupsilon_{k-1,c}-\frac{1}{2}\upepsilon}\cdot\prod_{b=1}^{r_k}\frac{\upupsilon_{k,a}-\upupsilon_{k,b}-\upepsilon}{\upupsilon_{k,a}-\upupsilon_{k,b}+\upepsilon}\cdot\prod_{d=1}^{r_{k+1}}\frac{\upupsilon_{k,a}-\upupsilon_{k+1,d}+\frac{1}{2}\upepsilon}{\upupsilon_{k,a}-\upupsilon_{k+1,d}-\frac{1}{2}\upepsilon}=1\,.
\label{eq:XXXBetheExplicit}
\end{equation}
Geometrically, we identify $\C^*$ with an infinite cylinder of radius
$R^{-1}$ and view this cylinder as the base space of our twisted
$q$-oper.  We then send the radius to infinity, thereby arriving at a
twisted version of the discrete opers of Mukhin and
Varchenko~\cite{Mukhin_2005}.

The second limit takes us from the XXX spin chain to the Gaudin
model.  In order to do this, we set
$\kappa_i=e^{\upepsilon\upkappa_i}$, and let $\upepsilon$ go to $0$.
As expected, we obtain the Bethe equations for the inhomogeneous
Gaudin model, i.e., the higher rank analogues of \eqref{gbetheir}:
\begin{equation}\label{eq:rGaudinBetheExplicit}
\upkappa_{k+1}-\upkappa_{k}+\sum\limits_{s=1}^L\frac{l_s^k}{\upupsilon_{k,a}-\upsigma_s}
+\sum\limits_{c=1}^{r_{k-1}} \frac{1}{\upupsilon_{k,a}-\upupsilon_{k-1,c}}-\sum\limits_{b\neq a}^{r_k}\frac{2}{\upupsilon_{k,a}-\upupsilon_{k,b}}+\sum\limits_{d=1}^{r_{k+1}}\frac{1}{\upupsilon_{k,a}-\upupsilon_{k+1,d}}=0\,.
\end{equation}
 Note that the difference of the twists $\upkappa_i$ can be identified with the monodromy data of the connection $A(z)$ at infinity.

We have thus established the following hierarchy between integrable
spin chain models and oper structures. 
\begin{center}
\begin{figure}[!h]
\includegraphics[scale =0.4]{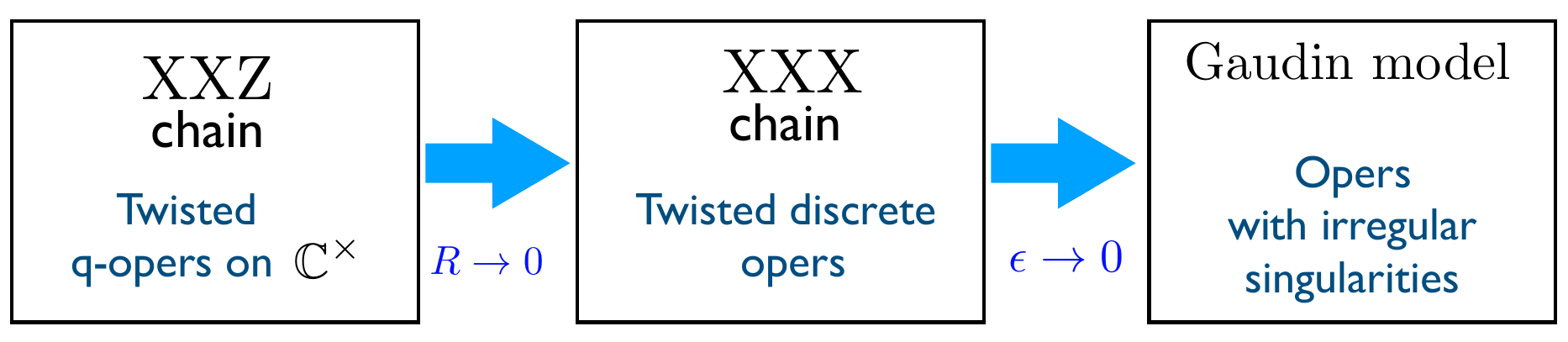}
\label{fig:hierarchy}
\end{figure}
\end{center}

We are currently working on extending this picture to include the XYZ
spin chain in this hierarchy.  We expect that the appropiate
difference opers to consider are ``twisted elliptic opers'' on elliptic
curves.

 \section{Quantum $K$-theory of Nakajima quiver varieties and $q$-opers}\label{Sec:Kth}
\subsection{The quantum $K$-theory ring for partial flag varieties}
As we discussed in the introduction, integrable models play an
important role in enumerative geometry.  For example, consider the XXZ
spin chain for $\sl_N$ where the dominant weights at the marked points
$z_m$ all correspond
to the defining representation, i.e., $\lambda_m=(1,0,0,\dots, 0)$.
Recall that cotangent bundles to partial flag varieties are
particular case of quiver varieties of type $A$
(see \figref{fig:ANquiver}).
\begin{center}
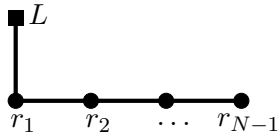
\begin{figure}[!h]
\begin{tikzpicture}
\draw [ultra thick] (0,0) -- (3,0);
\draw [ultra thick] (0,1) -- (0,0);
\draw [fill] (0,0) circle [radius=0.1];
\draw [fill] (1,0) circle [radius=0.1];
\draw [fill] (2,0) circle [radius=0.1];
\draw [fill] (3,0) circle [radius=0.1];
\node (1) at (0.1,-0.3) {$r_{1}$};
\node (2) at (1.1,-0.3) {$r_2$};
\node (3) at (2.1,-0.3) {$\ldots$};
\node (4) at (3.1,-0.3) {$r_{N-1}$};
\fill [ultra thick] (0-0.1,1) rectangle (0.1,1.2);
\node (5) at (0.3,1.15) {$L$};
\end{tikzpicture}
\caption{The cotangent bundle to the partial flag variety $T^*\mathbb{F}l_{\bm{\mu}}$}
\label{fig:ANquiver}
\end{figure}
\end{center}
It follows from work of Nakajima~\cite{Nakajima:2001qg} that the space
of localized equivariant $K$-theory of such a cotangent bundle can be
identified with an appropriate weight space in the corresponding XXZ
model; moreover, the span of all such weight spaces for partial flag
varieties of $\SL(N)$ is endowed with a natural action of the quantum
group $U_{q}(\sl_N)$.

In \cite{Koroteev:2017aa}, it was established that the Bethe algebra
for this XXZ model---the algebra generated by the $Q$-operators of the
XXZ spin chain---can be entirely described in terms of enumerative
geometry.  
The equivariant quantum $K$-theory\footnote{The quantum $K$-theory discussed here is the so-called
{\it PSZ quantum $K$-theory} \cite{Pushkar:2016qvw}, involving
quasimaps to Nakajima varieties, as opposed to the Givental-Lee
approach which relies on stable maps.   We refer to
\cite{Smirnov:2020lhm} for some recent progress in comparing the two approaches.}  of the cotangent bundle
to a partial flag variety has generators which are quantum versions of
tautological bundles.  It is shown in \cite{Koroteev:2017aa} that the
eigenvalues of these quantum tautological bundles are the symmetric
functions in the Bethe roots, so that the quantum $K$-theory may be
identified with the Bethe algebra.  Moreover, the twist parameters
$\kappa_{i+1}/\kappa_{i}$ and the inhomogeneity (or evaluation)
parameters $z_m$ are identified with the K\"ahler parameters of the
quantum deformation and the equivariant parameters respectively.



In the case of complete flag varieties, the authors of
\cite{Koroteev:2017aa} found another set of generators which allows the
identification of the quantum $K$-theory ring with the algebra of functions
on a certain Lagrangian subvariety in the phase space for the
trigonometric Ruijsenaars-Schneider model.  The formulas used to
establish this (see Proposition 4.4 of \cite{Koroteev:2017aa}) are
strikingly similar to the equations \eqref{eq:MiuraQOperCond}
describing nondegenerate twisted $q$-opers.  Let us normalize the
section $s(z)$ in the definition of a twisted $q$-oper so that all
of its components are monic polynomials:
\begin{equation}\label{eq:sectionequation}
s_a (z) = \prod_{i=1}^{\rho_a}(z-w_{a,i})\,,\qquad a=1,\dots, N.
\end{equation}  If we restrict to the space of $q$-opers for which all
these polynomials have degree one, then their roots may be viewed as
coordinates.  These coordinates may be identified with the momenta of
the dual tRS model whereas the coordinates of the tRS model correspond
bijectively to the twist (K\"ahler) parameters $\kappa_{i+1}/\kappa_{i}$ \cite{Koroteev:2017aa}.


We remark that that the Bethe algebra of the XXX model has a similar
enumerative description: it is the equivariant quantum cohomology ring of the
cotangent bundle of a partial flag variety~\cite{GORBOUNOV201356}.
This led
Rimanyi, Tarasov and Varchenko to conjecture the analogous statement for the XXZ Bethe algebra~\cite[Conjecture 13.15]{Rimanyi:2014ef}.  Moreover, they
found a particular set of generators for the Bethe algebra, involving
determinantal formulas, which for complete flag varieties can be
identified with the generators for the quantum $K$-theory ring
discussed in the previous paragraph.

Their generators can be given a geometric interpretation as
coordinates on an appropriate space of $q$-opers, and this leads to
the following description of the quantum $K$-theory ring:



\begin{theorem}
  Let $X$ be the cotangent bundle of the $\GL(L)$ partial flag variety
  $T^*\mathbb{F}l_{\bm{\mu}}$ labeled by the vector
  $\bm{\mu}=(r_{N-1}-r_{N-2},\dots,r_1-r_2, L-r_{1})$ where
  $r_1,\dots,r_{N-1}$ and $L$ are the dimensions of the vector spaces
  corresponding to the nodes of the $A_{N-1}$ quiver and the framing on
  the first node in \figref{fig:ANquiver} respectively.  Let $T$ be a
  maximal torus in $\GL(L)$.

  Then the $T$-equivariant quantum $K$-theory of $X$ is given by the
  algebra
\begin{equation}
QK_T(X)=\frac{\mathbb{C}\left[\bm{p}^{\pm
      1},\bm{\kappa}^{\pm 1},\bfa^{\pm 1},q^{\pm\frac{1}{2}}\right]}{\left(\det M(z)=\det V_{1,\dots, N} \cdot\Pi(z)^{(\frac{1-N}{2})}\right)}\,,
\label{eq:QKTideal}
\end{equation}
where $\bm{\kappa} = (\kappa_1,\dots,\kappa_N)$ are the quantum
deformation parameters, $\bm{a}=(a_1,\dots,a_L)$ are the equivariant
parameters of the action of $T$ on $X$, 
\begin{equation}
\bm{p}=\{p_{a,i}\}\,,\quad  i=1,\dots, \rho_a,\quad a=1,\dots, N-1\,,
\end{equation}
are the coefficients of the polynomials
\begin{equation}
s_a (z) = \prod_{i=1}^{\rho_a}(z-w_{a,i})= \sum_{i=0}^{\rho_a}(-1)^i p_{a,i}\, z^{N-i} \,,
\label{eq:smonicNorm}
\end{equation}
where $\rho_k=r_k-r_{k-1}$, $\Pi(z) = \prod\limits_{s=1}^L (z-a_s)$,
and the matrix $M$ is given by
\begin{equation}
  M=\begin{pmatrix} \,  s_{1}^{(\small{\frac{1-N}{2})}} & \kappa_{1} s_{1}^{(\small{\frac{3-N}{2})}} & \cdots & \kappa_{1}^{N-1} s_{1}^{{(\small{\frac{N-1}{2})}}} \\ \vdots & \vdots & \ddots & \vdots \\  s_{N}^{(\small{\frac{1-N}{2})}}  & \kappa_{N} s_{N}^{(\small{\frac{3-N}{2})}} & \cdots & \kappa_{N}^{N-1} s_{N}^{(\small{\frac{N-1}{2})}} \, \end{pmatrix}.
\label{eq:MiuraQOperCond1q7}
\end{equation}

\end{theorem}

The ideal in \eqref{eq:QKTideal} depends on the auxiliary variable
$z$, and both sides of the equation are polynomials of degree $L$ in
$z$.  Thus, the quantum $K$-theory ring is determined by $L$ relations.

The case where $X$ is a complete flag variety, so that $L = N$ and
$\rho_1=\dots =\rho_{N-1}=1$, was investigated in
\cite{Koroteev:2017aa}.  Here, the determinantal relation in
\eqref{eq:QKTideal} yields the equations of motion of the $N$-body
trigonometric Ruijsenaars-Schneider model.

We would like to emphasize that the space of $q$-opers which is
described by the system of equations \eqref{eq:MiuraQOperCond}
contains the $K$-theory of $X$ \eqref{eq:QKTideal} as a subspace. In
particular, one identifies the singularities $z_1,\dots,z_L$ of the
$q$-oper with the equivariant parameters $a_1,\dots, a_L$ of the
action of the maximal torus of $\GL(L)$ on $X$, so that
$\Pi=W_{N}=\Lambda_1$.  (For $s>1$, $l^k_s=0$, so $\Lambda_s=1$.)

\begin{proof}
  We prove this by combining two theorems. First, we will use Theorem
  3.4 in \cite{Koroteev:2017aa}, where the quantum $K$-theory of
  Nakajima quiver varieties was defined using quasimaps
  \cite{Ciocan-Fontanine:2011tg,Okounkov:2015aa} from the base curve
  of genus zero to the quiver variety.  The second ingredient is
  Theorem \ref{eq:MainTheorem} from this paper in the special case
  when the dominant weights at all oper singularities correspond to
  the defining representation, so that $l_s^1=1$ for all $s$ and the
  other $l_s^k$ vanish.  Here, the Bethe ansatz equations
  \eqref{eq:BetheExplicit} are given by
\begin{align}\label{eq:BetheExplicit13}
\frac{\kappa_2}{\kappa_{1}}\prod_{s=1}^L \frac{u_{1,a}-a_s}{q^{-1}\,u_{1,a}-a_s}\cdot\prod_{b=1}^{r_1}\frac{q^{-1} u_{1,a}-u_{1,b}}{qu_{1,a}- u_{1,b}}\cdot\prod_{d=1}^{r_{2}}\frac{q^{\frac{1}{2}}u_{1,a}-u_{2,d}}{q^{-\frac{1}{2}} u_{1,a}-u_{2,d}}&=1\,,\cr
\frac{\kappa_{k+1}}{\kappa_{k}}\prod_{c=1}^{
 r_{k-1}}\frac{q^{\frac{1}{2}}u_{k,a}-u_{k-1,c}}{q^{-\frac{1}{2}} u_{k,a}-u_{k-1,c}}\cdot\prod_{b=1}^{r_k}\frac{q^{-1} u_{k,a}-u_{k,b}}{qu_{k,a}- u_{k,b}}\cdot\prod_{d=1}^{r_{k+1}}\frac{q^{\frac{1}{2}}u_{k,a}-u_{k+1,d}}{q^{-\frac{1}{2}} u_{k,a}-u_{k+1,d}}&=1\,,\\
\frac{\kappa_{N}}{\kappa_{N-1}} \prod_{c=1}^{r_{N-2}}\frac{q^{\frac{1}{2}}u_{N,a}-u_{N-1,c}}{q^{-\frac{1}{2}} u_{N,a}-u_{N-1,c}}\cdot\prod_{b=1}^{r_{N-1}}\frac{q^{-1} u_{N-1,a}-u_{N-1,b}}{qu_{N-1,a}- u_{N-1,b}}&=1\,,\notag
\end{align}
where $r_k=\rho_1+\dots+\rho_k$ and $k$ runs from $2$ to $N-2$ in the
middle equation.

The system \eqref{eq:BetheExplicit13} coincides with the Bethe
equations from Theorem 3.4 of \cite{Koroteev:2017aa} up to the
identification of Bethe roots and twists.  This latter set of Bethe
equations describes the relations in the quantum $K$-theory of $X$,
where the Bethe roots $v_{k,a}$ are the Chern roots of the $k$-th
tautological bundle over $X$ and the other variables are identified
with the geometry of $X$ as in the statement of the theorem.

We have proven in Theorem \ref{eq:MainTheorem} that equations
\eqref{eq:BetheExplicit13} can be written as \eqref{eq:BetheDetForm}.  For
$k=N$ and for the dominant weights above, this gives
\begin{equation}
\det  M_{1,\dots, N}(z)=\det V_{1,\dots, N} W_N(z)^{(\frac{1-N}{2})}=\det V_{1,\dots, N} \Pi(z)^{(\frac{1-N}{2})}\,.
\end{equation}
This statement completes the proof.
\end{proof}

\subsection{The trigonometric RS model in the dual frame}
The trigonometric Ruijsenaars-Schneider model enjoys \emph{bispectral
  duality}. This may be described in geometric language as follows.
For a given quiver variety $X$ of type $A$, there are two dual
realizations of the tRS model.  The first was explained for $\SL(2)$
in Section~\ref{s:trssl2}.  Here, the twist variables
$\bm{\kappa}$ play the role of particle positions; their
conjugate momenta
$p_{\bm{\kappa}}=(p_{\kappa_1},\dots,p_{\kappa_N})$ are
defined as
\begin{equation}
p_{\bm{\kappa}}=\exp\left(\frac{\partial \mathcal{Y}}{\partial \log\bm{\kappa}}\right)\,,
\end{equation}
where $\mathcal{Y}$ is the so-called Yang-Yang function which depends
on the Bethe roots $v_{k,a}$ as well as all other parameters. The
Yang-Yang function serves as a potential for equations
\eqref{eq:BetheExplicit13} \cite{Nekrasov:2009uh,Nekrasov:2009ui},
i.e., the $k$-th equation is given by 
\begin{equation}
\exp\left(\frac{\partial \mathcal{Y}}{\partial \log v_{k,a}}\right)=1\,,\qquad a= 1,\dots r_k\,,\quad k = 1,\dots, N-1\,.
\end{equation}
(See
\cite{Gaiotto:2013bwa,Bullimore:2015fr,Koroteev:2017aa} for more details.)

The other realization---the \emph{3d-mirror} or
\emph{spectral/symplectic dual} description---involves a mirror quiver
variety $X^\vee$ and the associated dual Yang-Yang function
$\mathcal{Y}^\vee$.  (For a mathematical introduction,
see~\cite{Nakajima:2016}.  The construction of the mirror is discussed in
\cite{Gaiotto:2013bwa}.) Under the mirror map, the
K\"ahler parameters $\bm{\kappa}$ are interchanged with the
equivariant parameters $\bm{a}$; the same holds for the conjugate
momenta $p_{\bm{\kappa}}$ and $p_{\bm{a}}$.  Therefore, the variables
$\bm{a},p_{\bm{a}}$ can be viewed as the canonical degrees of freedom
in the dual tRS model; this has been studied in the context of
enumerative geometry in \cite{korzeit,Zabrodin:}. In particular, such
a duality was demonstrated between the XXZ spin chain whose Bethe
equations describe the equivariant quantum $K$-theory of the quiver
variety from \figref{fig:ANquiver} and the $L$-body tRS model whose
coordinates are the equivariant parameters $(a_1,\dots,a_L)$ of the
maximal torus for $\GL(L)$.  This result allows us to construct a
natural embedding of the intersection of two Lagrangian cycles inside
the tRS phase space into the space of $q$-opers with the first
fundamental weight at each regular singularity.

\bibliography{cpn12}
\end{document}